\documentclass[12pt,reqno]{amsart}

\usepackage[utf8]{inputenc}

\usepackage{amsmath} 
\usepackage{amsthm} 
\usepackage{amssymb} 
\usepackage{amscd} 

\usepackage{enumitem} 
\usepackage{mathtools} 
\usepackage{mathrsfs} 
\usepackage{hyperref} 
\hypersetup{
    linktocpage = true,
    colorlinks=true, 
    linkcolor=red,
}

\newtheorem{theorem}{Theorem}[section]
\newtheorem*{theorem*}{Theorem}
\newtheorem{proposition}[theorem]{Proposition}
\newtheorem{lemma}[theorem]{Lemma}

\theoremstyle{definition}

\theoremstyle{remark}
\newtheorem{remark}[theorem]{Remark}
\newtheorem*{remark*}{Remark}

\newcommand{\con}[1]{\mathbb{#1}}
\newcommand{\R}{\con{R}} 
\newcommand{\norm}[1]{\left | \left |{#1} \right | \right |}
\newcommand{\seminorm}[1]{\left [ {#1} \right ] }

\newcommand{\fraclaplacian}{(-\Delta)^s}
\renewcommand{\d}{\,\mathrm{d}} 
\newcommand{\dx}{\,\mathrm{d}x} 
\newcommand{\bpar}[1]{\left ( {#1}\right )}
\newcommand{\setcond}[2]{\left \{ #1 \ : \ #2  \right \}}
\newcommand\beqc[1]{\left\{\begin{array}{#1}}
\newcommand\eeqc{\end{array} \right.}
\def\PDEsystem{rcll}
\def\bmatrix{\begin{pmatrix}}
\def\ematrix{\end{pmatrix}}

\DeclareMathOperator{\tr}{tr}

\DeclareMathOperator{\dist}{dist}

\DeclareMathOperator{\PV}{P.V.}
\let\div\relax
\DeclareMathOperator{\div}{div}

\numberwithin{equation}{section}

\title[Radial stable solutions to fractional semilinear elliptic equations]{Regularity of radial stable solutions to semilinear elliptic equations for the fractional Laplacian}
\author{Tomás Sanz-Perela}
\address{T. Sanz-Perela:
Universitat Polit\`ecnica de Catalunya and BGSMath, Departament de Matem\`{a}tiques, Diagonal 647, 08028 Barcelona, Spain}

\email{tomas.sanz@upc.edu}

\thanks{The author is supported by MINECO grants MDM-2014-0445 and MTM2014-52402-C3-1-P. He is member of the Barcelona Graduate School of Mathematics and part of the Catalan research group 2014 SGR 1083. }


\begin{document}

\begin{abstract}
We study the regularity of stable solutions to the 
problem
\[
\left\{
\begin{array}{rcll}
(-\Delta)^s u &=& f(u) & \text{in} \quad B_1\,, \\
u &\equiv&0 & \text{in} \quad \mathbb \R^n\setminus B_1\,,
\end{array} \right.
\]
where $s\in(0,1)$. Our main result establishes an $L^\infty$ bound for stable and radially decreasing $H^s$ solutions to this problem
in dimensions $2 \leq n < 2(s+2+\sqrt{2(s+1)})$. In particular, 
this estimate holds for all $s\in(0,1)$ in dimensions $2 \leq n\leq 6$. It applies to all nonlinearities $f\in C^2$. 

For such parameters $s$ and $n$, our result leads to the regularity of the extremal solution
when $f$ is replaced by $\lambda f$ with $\lambda > 0$.
This is a widely studied question for $s=1$, which is still largely open in the nonradial case both for $s=1$ and $s<1$.
\end{abstract}

\maketitle

\section{Introduction}

This paper is devoted to the study of the regularity of stable solutions to the semilinear problem
\begin{equation}
\label{Eq:SemilinearProblemB1}
\beqc{\PDEsystem}
    \fraclaplacian u &= & f(u) & \text{in} \quad  B_1\,,\\
    u &= &0 & \text{in} \quad  \R^n\setminus B_1\,,
\eeqc
\end{equation} 
where $B_1$ is the unit ball in $\R^n$ and $f$ is a $C^2$ function. The operator $ \fraclaplacian$ is the fractional Laplacian, defined for $s\in(0,1)$ by
$$
\fraclaplacian u (x):= c_{n,s}  \PV \int_{\R^n} \dfrac{u(x) - u(z)}{|x-z|^{n+2s}} \d z\,,
$$
where $c_{n,s}>0$ is a normalizing constant depending only on $n$ and $s$ and $\PV$ stands for principal value.

Our results are motivated by the following problem, a variation of \eqref{Eq:SemilinearProblemB1}:
\begin{equation}
\label{Eq:ProblemLambda}
    \beqc{\PDEsystem}
    \fraclaplacian u &= &\lambda f(u) &  \text{in} \quad \Omega \,,\\
    u &= &0 &  \text{in} \quad   \R^n\setminus\Omega\,,
    \eeqc
\end{equation}
where $\Omega \subset \R^n$ is a bounded smooth domain, $\lambda > 0$ is a real parameter and the function $f:[0,\infty) \longrightarrow \R$ satisfies
\begin{equation}
\label{Eq:Conditionsf}
f\in C^1([0,\infty))\,, \ f\text{ is nondecreasing, } f(0)>0, \text{ and } \lim_{t \rightarrow +\infty} \dfrac{f(t)}{t} = + \infty \,.
\end{equation}
In this article we study \eqref{Eq:ProblemLambda} when $\Omega = B_1$.

It is well known (see \cite{RosOtonSerra-Extremal}) that, for $f$ satisfying \eqref{Eq:Conditionsf}, there exists a finite extremal parameter $\lambda^*$ such that, if  $0<\lambda<\lambda^*$ then problem \eqref{Eq:ProblemLambda} admits a minimal classical solution $u_{\lambda}$ which is stable ---see \eqref{Eq:StabilityCondition} below---, while for $\lambda > \lambda^*$ it has no solution, even in the weak sense. 
The family $\setcond{u_\lambda}{0<\lambda<\lambda^*}$ is increasing in $\lambda$ and its pointwise limit when $\lambda \nearrow \lambda^*$ is a weak solution of \eqref{Eq:ProblemLambda} with $\lambda=\lambda^*$. Such solution, denoted by $u^*$, is called  \emph{extremal solution} of \eqref{Eq:ProblemLambda}. 
As in \cite{RosOtonSerra-Extremal}, we say that $u$ is a weak solution of \eqref{Eq:ProblemLambda} when $ u \in L^1(\Omega)$, $f(u)\delta^s \in L^1(\Omega)$, where $\delta(x) = \dist (x, \partial \Omega)$, and
\begin{equation}
\label{Eq:DefinitionWeakSolution}
\int_\Omega u \fraclaplacian \zeta \d x = \int_\Omega \lambda f(u) \zeta \d x
\end{equation}
for all $\zeta$ such that $\zeta$ and $\fraclaplacian \zeta$ are bounded in $\Omega$ and $\zeta \equiv 0$ on $\partial \Omega$.

In the nineties, H.~Brezis and J.L.~Vázquez \cite{BrezisVazquez} raised the question of determining the regularity of $u^*$ depending on the dimension $n$ for the local version ($s=1$) of \eqref{Eq:ProblemLambda} ---see also the open problems raised by H.~Brezis in \cite{BrezisIFT}. This is equivalent to determine whether $u^*$ is bounded or unbounded. There are several results in this direction for the classical problem (see Remark~\ref{Remark:KnownResultsLaplacian} for more details and also the monograph~\cite{Dupaigne}). 

Regarding the problem for the fractional Laplacian, there are fewer results concerning the regularity of stable solutions and in particular of the extremal solution of \eqref{Eq:ProblemLambda}. This problem was first studied for the fractional Laplacian by X.~Ros-Oton and J.~Serra in \cite{RosOtonSerra-Extremal}. There, the authors proved the existence of the family of minimal and stable solutions $u_\lambda$, as well as the existence of the extremal solution $u^*$. They also showed that if $f$ is convex then $u^*$ is bounded whenever $n < 4s$, and that if $f$ is $C^2$ and $f f'' /(f')^2$ has a limit at infinity, the same happens if $n <10s$ (see Remark~\ref{Remark:KnownResultsLaplacian} for more comments on this). Later, X.~Ros-Oton~\cite{RosOton-Gelfand} improved this result in the case of the exponential nonlinearity $f(u) = e^u$, showing that $u^*$ is bounded whenever $n \leq 7$ for all $s \in (0,1)$. More precisely, the condition involving $n$ and $s$ that he found is the following:
\begin{equation}
\label{Eq:SharpConditionNGelfandProblem}
\dfrac{\Gamma(\frac{n}{2})\Gamma(1+s)}{\Gamma(\frac{n - 2s}{2})} > \dfrac{\Gamma^2(\frac{n + 2s}{4})}{\Gamma^2(\frac{n - 2s}{4})}\,.
\end{equation}
In particular, for $s \gtrsim 0.63237\dots$, $u^*$ is bounded up to dimension $n = 9$. 
As explained in Remark 2.2 of \cite{RosOton-Gelfand}, condition \eqref{Eq:SharpConditionNGelfandProblem} is expected to be optimal, since if \eqref{Eq:SharpConditionNGelfandProblem} does not hold, then $\log |x|^{-2s}$ is a singular extremal solution of the problem $\fraclaplacian u = \lambda e^u$ in all $\R^n$. Nevertheless, this is still an open problem, since this last example is not our Dirichlet problem in a bounded domain.

To our knowledge, \cite{RosOtonSerra-Extremal,RosOton-Gelfand} are the only papers where problem \eqref{Eq:ProblemLambda} is studied. However,  the article by A.~Capella, J.~Dávila, L.~Dupaigne, and Y.~Sire \cite{CapellaEtAl} deals with a similar problem to \eqref{Eq:ProblemLambda} but for a different operator, the spectral fractional Laplacian $A^s$ defined via the Dirichlet eigenvalues and eigenfunctions of the Laplace operator. It studies the problem of the extremal solution for the operator $A^s$ in the unit ball and it establishes that, if $2 \leq n < 2(s+2+\sqrt{2(s+1)})$, then $u^* \in L^\infty (B_1)$. In particular, $u^*$ is bounded in dimensions $2 \leq n \leq 6$ for all $s\in (0,1)$. In the present work, we use similar ideas to the ones in \cite{CapellaEtAl} to study the same problem in $B_1$, but now with $A^s$ replaced by the fractional Laplacian. We obtain the same condition on $n$ and $s$ guaranteeing regularity of the extremal solution to \eqref{Eq:ProblemLambda}. Moreover, in the arguments of \cite{CapellaEtAl} there are two points where an estimate is missing and hence the result is not completely proved. In this paper we establish such estimate (given in Proposition~\ref{Prop:HorizontalGradientEstimateRing}) which is valid for the fractional Laplacian and also for the spectral fractional Laplacian. Hence, we complete the proofs of \cite{CapellaEtAl} (see the comment before Remark~\ref{Remark:KnownResultsLaplacian} and also Remarks~\ref{Remark:MissingEstimateCapellaEtAl} and \ref{Remark:MissingTermCapellaEtAl}).

The following is our main result, concerning the boundedness of the extremal solution.

\begin{theorem}
\label{Th:RadialMainTheoremExtremal}
Let $n\geq 2$, $s\in (0,1)$, and $f$ be a $C^2$ function satisfying \eqref{Eq:Conditionsf}. Let $u^*$ be the extremal solution of \eqref{Eq:ProblemLambda} with $\Omega = B_1$, the unit ball of $\R^n$. Then, $u^*$ is radially symmetric and decreasing in $B_1\setminus\{0\}$ and we have that:
\begin{enumerate}[label=(\roman*)]
\item If $n < 2\bpar{s+2+\sqrt{2(s+1)}}$, then $u^* \in L^{\infty}(B_1)$.
\item  If $n \geq 2\bpar{s+2+\sqrt{2(s+1)}}$, then for every $\mu > n/2 - s - 1 - \sqrt{n-1}$, it holds $u^*(x) \leq C |x|^{-\mu}$ in $B_1$ for some constant $C>0$.
\end{enumerate}
\end{theorem}

As a consequence, $u^*$ is bounded for all $s\in (0,1)$ whenever $2 \leq n \leq 6$. The same holds if $n = 7$ and $s \gtrsim 0.050510\ldots$, if $n=8$ and $s \gtrsim 0.354248\ldots$, and if $n=9$ and $s \gtrsim 0.671572\ldots$. Note that the assumption in (i) never holds for $n\geq 10$. In the limit $s\uparrow 1$, the condition on $n$ in statement (i) corresponds to the optimal one for the local problem in the ball, that is $n<10$ ---see~\cite{CabreCapella-Radial}. Instead, for powers $s<1$, the hypothesis in (i) is not optimal: for the exponential nonlinearity $f(u) = e^u$ a better assumption is \eqref{Eq:SharpConditionNGelfandProblem} ---see~\cite{RosOton-Gelfand}.

Theorem~\ref{Th:RadialMainTheoremExtremal} is a consequence of the  stability of $u^*$. We say that a weak solution $u\in L^1(\Omega)$ of \eqref{Eq:ProblemLambda} is \emph{stable} if 
\begin{equation}
\label{Eq:StabilityCondition}
\lambda \int_\Omega f'(u) \xi^2 \dx \leq \seminorm{\xi}_{H^{s}(\R^n)}^2 := \int_{\R^n} |(-\Delta)^{s/2}\xi|^2 \dx 
\end{equation}
for all $\xi\in H^{s}(\R^n)$ such that $\xi \equiv 0$ on $\R^n \setminus\Omega$.
Note that the integral in the left-hand side of \eqref{Eq:StabilityCondition} is well defined if $f$ is nondecreasing, an assumption that we make throughout all the paper.

In case of problem \eqref{Eq:ProblemLambda}, all the solutions $u_\lambda$ with $\lambda < \lambda^*$, as well as the extremal solution, are stable. This property follows from their minimality.
When $u\in H^s(\R^n)$, stability is equivalent to the nonnegativeness of the second variation of the energy associated to \eqref{Eq:ProblemLambda} at $u$.

The proof of Theorem~\ref{Th:RadialMainTheoremExtremal} is based only on the stability of solutions. First, we show that bounded stable solutions are radially symmetric and monotone (see Section~\ref{Sec:RadialSymmetry}). Then, we use this, the stability condition and the equation to prove our estimates. This procedure is first applied to $u_\lambda$, with $\lambda < \lambda^*$, which are bounded stable solutions and thus regular enough, and we establish some estimates that are uniform in $\lambda < \lambda^*$. More precisely, they depend essentially on $\norm{u_\lambda}_{L^1(\R^n)}$, a quantity that can be bounded independently of $\lambda$ ---see Remark~\ref{Remark:UniformEstimatesInLambda} for more details about this fact. Once we have these uniform estimates, we can pass to the limit $\lambda \to \lambda^*$ and use monotone convergence to prove the result for $u^*$.

This result, Theorem~\ref{Th:RadialMainTheoremExtremal}, is a consequence of the following more general statement, which applies to the class of stable and radially decreasing $H^s$ weak solutions ---not necessarily bounded--- to \eqref{Eq:SemilinearProblemB1}. Recall that our notion of weak solution is given in \eqref{Eq:DefinitionWeakSolution}. Recall also (see Section~\ref{Sec:RadialSymmetry}) that positive bounded stable solutions to \eqref{Eq:SemilinearProblemB1} will be shown to be radially decreasing in $B_1$.

\begin{theorem}
	\label{Th:RadialMainTheorem}
	Let $n\geq 2$, $s\in (0,1)$, and $f$ be a $C^2$ nondecreasing function. Let $u \in H^s(\R^n)$ be a stable radially decreasing weak solution to \eqref{Eq:SemilinearProblemB1}. We have that:
	\begin{enumerate}[label=(\roman*)]
		\item If $n < 2\bpar{s+2+\sqrt{2(s+1)}}$, then $u\in L^{\infty}(B_1)$. Moreover,
		$$
		\norm{u}_{L^\infty (B_1)} \leq C
		$$
		for some constant $C$ that depends only on $n$, $s$, $f$ and
		 $\norm{u}_{L^1 (\R^n)}$.
		\item  If $n \geq 2\bpar{s+2+\sqrt{2(s+1)}}$, then for every $\mu > n/2 - s - 1 - \sqrt{n-1}$, it holds
		$$
		u(x) \leq C |x|^{-\mu} \quad \text{ in } B_1\,,
		$$ 
		for some constant $C$ that depends only on $n$, $s$, $\mu$, $f$ and
		$\norm{u}_{L^1 (\R^n)}$.
	\end{enumerate}
	
\end{theorem}

A main tool used in the present article is the extension problem for the fractional Laplacian, due to L.~Caffarelli and L.~Silvestre \cite{CaffarelliSilvestre}. Namely,
for $s \in (0,1)$ and given a function $u:\R^n \rightarrow \R$, consider $v$ the solution of
 \begin{equation}
\label{Eq:ExtensionProblemDirichlet}
     \beqc{rcll}
    \div (y^a \nabla v ) &= & 0 & \text{in } \R^{n+1}_+ \,,\\
    v\, &= &u &  \text{on } \partial \R^{n+1}_+ = \R^n\,,
    \eeqc
\end{equation}
where $a = 1-2s$ and $\R^{n+1}_+ = \setcond{(x,y) \in \R^{n+1}}{x \in \R^n, \ y \in (0, +\infty)}$. As it is well known (see \cite{CaffarelliSilvestre}), the limit $- \lim_{y\downarrow 0} y^a \partial_y v$ agrees with $\fraclaplacian u$ up to a positive multiplicative constant. We will refer to the solution of \eqref{Eq:ExtensionProblemDirichlet}, $v$, as the $s$-\textit{harmonic extension} of $u$. This terminology is motivated by the fact that, when $s=1/2$, then $a=0$ and $v$ is the harmonic extension of $u$.

Throughout the paper, $(x,y)$ denote points in $\R^n \times (0, +\infty) = \R^{n+1}_+$. We also use the notation $$\dfrac{\partial v }{ \partial \nu^a} = -\lim_{y \downarrow 0} y^a v_y$$ for the conormal exterior derivative and we will always assume the relation 
$$a = 1-2s \in (-1,1)\,.$$
Moreover, we denote by 
$$
\rho=|x|\quad \textrm{ and }r = \sqrt{\rho^2 + y^2}
$$
the modulus in $\R^n$ and $\R^{n+1}_+$, respectively. Therefore, $v_\rho$ will denote the derivative of $v$ in the horizontal radial direction, that is
$$
v_\rho (x,y) =  \frac{x}{\rho} \cdot \nabla_x v (x,y)  \quad \textrm{ with } \rho=|x|.
$$
We will always use the letter $u$ to denote a function defined in $\R^n$ and the letter $v$ for its $s$-harmonic extension in $\R^{n+1}_+$.

In \cite{CapellaEtAl}, the authors use also an extension problem for the spectral operator $A^s$. Indeed, one can see that the spectral fractional Laplacian can be realized as the boundary Neumann operator of a suitable extension in the half-cylinder $\Omega \times (0,+\infty)$. More precisely, one considers the extension problem
$$
\beqc{rcll}
    \div (y^a \nabla w ) &= & 0 & \text{in } \Omega \times (0,+\infty)\,,\\
    w &= &0 &  \text{on } \partial \Omega \times [0,+\infty)\,, \\
    w &= &u &  \text{on } \Omega \times \{0\}\,,
    \eeqc
$$
with $a = 1 - 2s$.
Then, it can be proven that $-\lim_{y \downarrow 0} y^a w_y$ agrees with $A^s u$ up to a multiplicative constant. Notice that the solution $w$ (extended by $0$ to all $\overline{\R^{n+1}_+}$) is a subsolution of \eqref{Eq:ExtensionProblemDirichlet} and thus, thanks to the maximum principle, one can use the Poisson formula for \eqref{Eq:ExtensionProblemDirichlet} to obtain estimates for $w$. This is what is done in \cite{CapellaEtAl} and suggested that similar arguments could be carried out for the fractional Laplacian, as we indeed do.

The proof of Theorem~\ref{Th:RadialMainTheorem} is mostly based on two ideas. First, by the representation formula for the fractional Laplacian, we see that the $L^{\infty}$ norm of a solution $u$ can be bounded by the integral over $B_1$ of $f(u) /|x|^{n - 2s}$ (see Lemma~\ref{Lemma:Boundu(x)Dirichlet}). Thus, it remains to estimate this integral. We bound it in $B_1 \setminus B_{1/2}$ using that the solution is radially decreasing (see Section~\ref{Sec:RadialSymmetry}). Regarding the integral in $B_{1/2}$, we can relate it with
$$
\int_{B_{1/2}\times (0,1)} y^a r^{-(n + 2 - 2s)}\rho v_\rho \d x \d y + \int_{B_{1/2}\times (0,1)} y^a r^{-(n + 2 - 2s)} y v_y \d x \d y\,,
$$
after an integration by parts in $B_{1/2}\times (0,1) \subset \R^{n+1}_+$ and seeing $f(u)$ as the flux $d_s \partial_{\nu ^a}v$ ---the other boundary terms are estimated using the results of Section~\ref{Sec:PreliminaryResults}. On the one hand, the integral involving $v_y$ can be absorbed in the left-hand side of the estimates by using the identity given in  Lemma~\ref{Lemma:Identityv_yMagicConstant} (see Section~\ref{Sec:ProofMainTheorem} for the details). On the other hand, the integral involving $v_\rho$ can be estimated, after using the Cauchy-Schwarz inequality, thanks to the next key proposition. It provides an estimate for a weighted Dirichlet integral involving the $s$-harmonic extension of stable solutions to \eqref{Eq:SemilinearProblemB1}. 

\begin{proposition}
\label{Prop:IntegrabilityHalfBall}
Let $n \geq 2$, $s \in (0,1)$, and $f$ be a nondecreasing $C^2$ function. Let $u\in H^s(\R^n)$ be a stable radially decreasing solution of \eqref{Eq:SemilinearProblemB1} and $v$ be its $s$-harmonic extension as in \eqref{Eq:ExtensionProblemDirichlet}. Assume that $\alpha$ is any real number satisfying
\begin{equation}
    \label{Eq:AlphaCondition}
    1 \leq \alpha < 1 + \sqrt{n-1}\,.
\end{equation}

Then,
\begin{equation}
\label{Eq:UniformBoundHalfBall}
\int_0^{\infty}
\int_{B_{1/2}} y^a v_\rho^2 \rho^{-2\alpha} \d x \d y \leq C\,,
\end{equation}
where $C$ is a constant depending only on $n$, $s$, $\alpha$, $\norm{u}_{L^{1}(B_1)}$, $\norm{u}_{L^{\infty}(B_{7/8}\setminus B_{3/8})}$, and $\norm{f'(u)}_{L^{\infty}(B_{7/8}\setminus B_{3/8})}$.
\end{proposition}

The key point to establish Proposition~\ref{Prop:IntegrabilityHalfBall} ---as well as its analogous in \cite{CapellaEtAl}--- is the particular choice of the test function $\xi$ in the stability condition \eqref{Eq:StabilityConditionExtension}, which is equivalent to \eqref{Eq:StabilityCondition} when considering the extension to $\R^{n+1}_+$ of functions defined in $\R^n$. We take
\begin{equation}
\label{Eq:TestFunctionStability}
\xi = \rho^{1-\alpha}v_\rho \zeta \,,
\end{equation}
where $\alpha$ satisfies \eqref{Eq:AlphaCondition}, $v_\rho$ is the horizontal radial derivative of $v$, and $\zeta$ is a cut-off function. This choice, after controlling a number of integrals, will lead to \eqref{Eq:UniformBoundHalfBall}. A similar idea was already used by X.~Cabré and A.~Capella in \cite{CabreCapella-Radial} to prove the boundedness of $u^*$ in the radial case for the classical Laplacian, and later by A. ~Capella, J.~Dávila, L.~Dupaigne and Y.~Sire in \cite{CapellaEtAl} for $A^s$. 

Furthermore, another important ingredient in order to establish Theorem~\ref{Th:RadialMainTheorem} and Proposition~\ref{Prop:IntegrabilityHalfBall} is a crucial estimate for the $s$-harmonic extension of solutions to \eqref{Eq:SemilinearProblemB1}. In Proposition~\ref{Prop:HorizontalGradientEstimateRing} we establish such estimate, whose proof was missing in \cite{CapellaEtAl}, as mentioned before. It controls pointwise the horizontal gradient of $v$, where $v$ is the $s$-harmonic extension of $u$, in a cylindrical annulus about the origin.

\begin{remark}
\label{Remark:KnownResultsLaplacian}
The local version of problem \eqref{Eq:ProblemLambda} was first studied in the seventies and eighties, essentially for the exponential and power nonlinearities. When $f(u)= e^u$, it is known that $u^* \in L^{\infty}(\Omega)$ if $n\leq 9$ (see \cite{CrandallRabinowitz}), while $u^*(x) = \log |x|^{-2} $ when $\Omega = B_1$ and $n\geq 10$ (see \cite{JosephLundgren}). Similar results hold for $f(u) = (1 +u)^p$, and also for functions $f$ satisfying a limit condition at infinity (see \cite{Sanchon}). This is extended to the nonlocal case in \cite{RosOtonSerra-Extremal}, where the condition $n \leq 9$ becomes now $n < 10s$.

For the local case and general nonlinearities, the first result concerning the boundedness of the extremal solution was obtained by G.~Nedev \cite{Nedev}, who proved that $u^*$ is bounded in dimensions $n\leq3$ whenever $f$ is convex. The result in \cite{RosOtonSerra-Extremal} for $n < 4s$ also extends this to the nonlocal setting. 

Later, X.~Cabré and A.~Capella~\cite{CabreCapella-Radial} obtained an $L^\infty$ bound for $u^*$, when $s=1$ and $\Omega = B_1$, whenever $n\leq 9$. The best known result at the moment for general $f$ and $s=1$ is due to X.~Cabré~\cite{Cabre-Dim4}, and states that in dimensions $n \leq 4$ the extremal solution is bounded for every convex domain $\Omega$. This result was extended by S.~Villegas \cite{Villegas} to nonconvex domains. Nevertheless, the problem is still open in dimensions $5\leq n \leq 9$.
\end{remark}

As mentioned before, to our knowledge the only articles dealing with problem \eqref{Eq:ProblemLambda} are \cite{RosOtonSerra-Extremal} and \cite{RosOton-Gelfand}. There, the authors work in $\R^n$ and do not use the extension problem for the fractional Laplacian. For this reason, we include in the appendix of this article an alternative proof ---which uses the extension problem--- of the result of X.~Ros-Oton and J.~Serra \cite{RosOtonSerra-Extremal} that establishes the boundedness of the extremal solution in dimensions $n<10s$ in any domain when $f(u) = e^u$. This is Proposition~\ref{Prop:GelfandProblem10s} below.

The paper is organized as follows. Section~\ref{Sec:ExtensionProblem} is devoted to recall some results concerning the extension problem for the fractional Laplacian, as well as to express the stability condition using the extension problem. In Section~\ref{Sec:PreliminaryResults}, we establish some preliminary results which are used in the following sections. Section~\ref{Sec:RadialSymmetry} focuses on the symmetry and monotonicity of bounded stable solutions. Proposition~\ref{Prop:IntegrabilityHalfBall} is proved in Section~\ref{Sec:WeightedIntegrability}, and Theorem~\ref{Th:RadialMainTheorem} in Section~\ref{Sec:ProofMainTheorem}. Finally, in Appendix~\ref{Sec:GelfandProblem} we give an alternative proof of the result of \cite{RosOtonSerra-Extremal} concerning the exponential nonlinearity.

\section{The extension problem for the fractional Laplacian}
\label{Sec:ExtensionProblem}

In this section we recall briefly some results concerning the extension problem for the fractional Laplacian. The main feature is the following well known relation: if $v$ is the solution of the extension problem \eqref{Eq:ExtensionProblemDirichlet}, then
\begin{equation}
\label{Eq:DirichletToNeumannRelation}
\fraclaplacian u = \fraclaplacian \{v(\cdot,0)\} = d_s \dfrac{\partial v}{\partial \nu^a}\,, 
\end{equation}
for a positive constant $d_s$ which only depends on $s$.

Hence,  given $s\in (0,1)$, a function $u$ defined in $\R^n$ is a solution of $\fraclaplacian u = h$ in $\R^n$ if, and only if, its $s$-harmonic extension in $\R^{n+1}_+$ solves the problem
\begin{equation}
\label{Eq:ExtensionProblemNeumann}
    \beqc{\PDEsystem}
    \div (y^a \nabla v ) &= & 0 & \text{in } \R^{n+1}_+ \,,\\
    \dfrac{\partial v}{\partial \nu^a}  &= & \dfrac{h}{d_s} &  \text{on }  \R^n\,.
    \eeqc
\end{equation}

Recall that for problem \eqref{Eq:ExtensionProblemDirichlet} we have an explicit Poisson formula: 
$$
v(x,y) = P * u = \int_{\R^n} P(x-z, y) u(z) \d z\,, \  \textrm{ where } \ P(x, y) = P_{n,a} \dfrac{y^{1-a}}{\bpar{|x|^2 + y^2}^{\frac{n+1-a}{2}}}
$$
and the constant $P_{n,s}$ is such that, for every $y> 0$, $\int_{\R^n}P(x, y)\d x = 1\,$.

The relation between $v$ and $-y^a v_y$ via a conjugate equation gives a useful formula for the $y$-derivative of the solution of \eqref{Eq:ExtensionProblemNeumann}.

\begin{lemma}[see \cite{CaffarelliSilvestre}]
\label{Lemma:PoissonFormulaConjugate}
Let $s \in (0,1)$, $a = 1 -2s$, $h:\R^n \rightarrow \R$ and $v$ be the solution of \eqref{Eq:ExtensionProblemNeumann}. Then, 
$$
-v_y(x,y) = \Gamma * \dfrac{h}{d_s} = \dfrac{1}{d_s}\int_{\R^n}  \Gamma(x-z, y) h (z) \d z\,.
$$
where
$$
\Gamma(x, y) = \Gamma_{n,s} \dfrac{y}{\bpar{|x|^2 + y^2}^{\frac{n+1+a}{2}}} =  \Gamma_{n,s} \dfrac{y}{\bpar{|x|^2 + y^2}^{\frac{n+2-2s}{2}}} \,,
$$
with a constant $\Gamma_{n,s}$ depending only on $n$ and $s$.
\end{lemma}

This is proved by considering the function $w = - y^a v_y$. A simple computation shows that $w$ solves the conjugate problem
$$
\beqc{\PDEsystem}
    \div (y^{-a} \nabla w ) &= & 0 & \text{in } \R^{n+1}_+ \,, \\
    w &= &h / d_s&  \text{on } \partial \R^{n+1}_+ = \R^n\,.
\eeqc
$$
Then, we use the Poisson formula for this problem to obtain
$$
- y^a v_y = w =  \dfrac{y^{1+a}}{d_s}\int_{\R^n}  P_{n,-a} \dfrac{h (z)}{\bpar{|x-z|^2 + y^2}^{\frac{n+1+a}{2}}} \d z\,.
$$

Recall also that the fundamental solution of the fractional Laplacian is well known. Namely, given $h:\R^n \rightarrow \R$ regular enough (for instance $h$ continuous with compact support), the unique continuous and bounded solution of $\fraclaplacian u = h $ in $\R^n$ is given by
$$
u(x) = C \int_{\R^n} \dfrac{h (z)}{|x-z|^{n-2s}} \d z \,,
$$
for a constant $C$ which depends only on $n$ and $s$ (see \cite{CaffarelliSilvestre,CabreSireI}). Using this last formula and the maximum principle, we easily deduce a useful pointwise bound for solutions of the Dirichlet problem for the fractional Laplacian. It is given by the following lemma.

\begin{lemma}
\label{Lemma:Boundu(x)Dirichlet}
Let $\Omega \subset \R^n$ a bounded smooth domain, $s\in (0,1)$ and $h:\Omega \rightarrow \R$ a nonnegative bounded function. Let $u \in H^s(\R^n)$ be a weak solution of the Dirichlet problem
$$
	\beqc{\PDEsystem}
	\fraclaplacian u &= & h & \textrm{in } \Omega \,,\\
	u &= &0 &  \textrm{in } \R^n\setminus \Omega\,.
	\eeqc
$$

Then, for every $x\in \R^n$,
\begin{equation}
\label{Eq:Boundu(x)Dirichlet}
u(x) \leq C \int_{\Omega} \dfrac{h (z)}{|x-z|^{n-2s}} \d z \,,
\end{equation}
for a constant $C$ which depends only on $n$ and $s$.
\end{lemma}

This result is the analogous of Lemma 6.1 in \cite{CapellaEtAl} and is the first step in order to prove Theorem~\ref{Th:RadialMainTheorem}. Indeed, we will estimate the $L^\infty$ norm of a solution by controlling the right-hand side of \eqref{Eq:Boundu(x)Dirichlet}, which can be related to the Dirichlet integral in \eqref{Eq:UniformBoundHalfBall} through an integration by parts (see Section~\ref{Sec:ProofMainTheorem}).

As mentioned in the introduction, the main property in which are based our estimates is stability. Recall that a solution of \eqref{Eq:ProblemLambda} is stable if it satisfies \eqref{Eq:StabilityCondition}. Since we want to work with the $s$-harmonic extension of such solutions, we need to rewrite the stability condition \eqref{Eq:StabilityCondition} in terms of the extension of functions in $\R^{n+1}_+$.

It is well known that the space $H^s(\R^n)$ coincides with the trace of $H^1( \R^{n+1}_+, y^a)$  on $\partial \R^{n+1}_+$ (see for instance \cite{Frank2013}). In particular, every function $\xi : \R^{n+1}_+ \rightarrow \R$ such that $\xi \in L^2_{\mathrm{loc}}(\R^{n+1}_+, y^a)$ and $\nabla \xi \in L^2(\R^{n+1}_+, y^a)$ has a trace on $H^{s}(\R^n)$ and satisfies the following inequality (see Proposition~3.6 in \cite{Frank2013}):
\begin{equation}
\label{Eq:TraceInequalityHalfSpace}
\seminorm{\tr\xi}_{H^s(\R^n)} \leq d_s \seminorm{\xi}_{H^1( \R^{n+1}_+, y^a)}\,,
\end{equation}
where we use the notation
$$
\seminorm{\varphi}_{H^1( \R^{n+1}_+, y^a)} = \int_{\R^{n+1}_+} y^a |\nabla \varphi|^2 \d x \d y
$$
and $d_s$ is the constant appearing in \eqref{Eq:DirichletToNeumannRelation}. In addition, $d_s$ is the optimal constant in \eqref{Eq:TraceInequalityHalfSpace}, as seen next. 

To show why $d_s$ is the optimal constant, we find a case where the equality is attained. Consider $w \in H^s(\R^n)$ and let $W$ denote the solution of
$$
\beqc{\PDEsystem}
    \div (y^a \nabla W ) &= & 0 & \text{in } \R^{n+1}_+\,,\\
    W &= & w &  \text{on } \R^n\,.
    \eeqc
$$
Notice that $W$ minimizes the seminorm $\seminorm{\cdot}_{H^1( \R^{n+1}_+, y^a)}$ among all functions whose trace on $\R^n$ is $w$, because it solves the Euler-Lagrange equation of the functional
$$
E(\varphi) = \int_{\R^{n+1}_+} y^a |\nabla \varphi|^2 \d x \d y\,.
$$
Therefore, integrating by parts and using that $d_s \dfrac{\partial W}{\partial \nu ^a} = \fraclaplacian w$ at $\R^n$, we have
\begin{align*}
d_s\seminorm{W}_{H^1( \R^{n+1}_+, y^a)} &= d_s\int_{\R_+^{n+1}} y^a |\nabla W |^2 \d x \d y \\ 
& = d_s \int_{\R^n} w   \dfrac{\partial W}{\partial \nu ^a} \d x \\
&= \int_{\R^n} w \fraclaplacian w \d x  \\ 
&= \int_{\R^n} (-\Delta)^{s/2} w (-\Delta)^{s/2} w \d x \\ &=\seminorm{w}_{H^s(\R^n)}\,.
\end{align*}
This shows that the optimal constant in \eqref{Eq:TraceInequalityHalfSpace} is $d_s$ and that the equality is achieved when one takes the $s$-harmonic extension of a function defined in $\R^n$.

Using \eqref{Eq:TraceInequalityHalfSpace}, we say that $u$ is a stable solution to
$$
\beqc{\PDEsystem}
\fraclaplacian u &= & f(u) &  \text{in} \quad   \Omega\,,\\
u &= &0 &  \text{in} \quad   \R^n\setminus \Omega \,.
\eeqc
$$ 
if
\begin{equation}
\label{Eq:StabilityConditionExtension}
  \int_\Omega f'(u) \xi^2 \dx \leq d_s\int_{\R_+^{n+1}} y^a |\nabla \xi |^2 \d x \d y \,,
\end{equation}
for every $\xi \in H^1( \R^{n+1}_+, y^a)$ such that its trace has compact support in $\Omega$. Notice that it is not necessary to take $\xi$ as the $s$-harmonic extension of its trace (that is, $\xi$ need not solve $\div (y^a \nabla \xi) = 0$ in $\R^{n+1}_+$). This gives us more flexibility for the choice of functions in the stability condition. However, if we  want an inequality completely equivalent to \eqref{Eq:StabilityCondition} ---in the sense that we do not lose anything when going to $\R^{n+1}_+$---, we need to consider always test functions solving $\div (y^a \nabla \xi) = 0$ in $\R^{n+1}_+$.

\section{Preliminary results: estimates for solutions of \eqref{Eq:SemilinearProblemB1}}
\label{Sec:PreliminaryResults}

The purpose of this section is to provide some estimates for solutions of \eqref{Eq:SemilinearProblemB1} that will be used in the subsequent sections. In particular, we give estimates for the derivatives of the $s$-harmonic extension of solutions to \eqref{Eq:SemilinearProblemB1}.

The three main estimates of this section are stated below. The first two results, Lemma~\ref{Lemma:Estimate for v_i and v_y} and Proposition~\ref{Prop:GradientEstimate}, concern the decay at infinity of $\nabla v$, where $v$ solves $\div(y^a \nabla v) = 0$ in $\R^{n+1}_+$. We control the decay at infinity since we deal with integrals in $\R^{n+1}_+$ weighted by $y^a$, with $a \in (-1,1)$, and $y^a$ is not integrable at infinity. In \cite{CapellaEtAl}, the authors use that the extension of solutions for the spectral fractional Laplacian, as well as their derivatives, have exponential decay as $y \rightarrow +\infty$. This allows them to overcome the problem of integrability at infinity. Instead, in the case of the fractional Laplacian, such exponential decay does not hold. Nevertheless, we establish a power decay in Lemma~\ref{Lemma:Estimate for v_i and v_y} and in Proposition~\ref{Prop:GradientEstimate}, and this will be enough for our purposes. The estimates we deduce in these two results are in terms of $u$, the trace of $v$ on $\R^n$, but we do not assume that $u$ solves any equation in $\R^n$. On the contrary, the third result of this section, Proposition~\ref{Prop:HorizontalGradientEstimateRing}, is an estimate up to $\{y=0\}$ and in this case we assume that $u$ is a solution to \eqref{Eq:SemilinearProblemB1}. 

Before presenting the three results of this section, let us make a comment on the right-hand sides of the estimates that we establish. We point out that the constants appearing in the statement of Lemma~\ref{Lemma:Estimate for v_i and v_y} depend on $\norm{u}_{L^\infty (B_1)}$ instead of $\norm{u}_{L^1(B_1)}$, in contrast with the other two main estimates of this section (Propositions~\ref{Prop:GradientEstimate} and \ref{Prop:HorizontalGradientEstimateRing}). This will cause no problem since  the lemma will be used only in Section~\ref{Sec:RadialSymmetry}, where we will assume that $u\in L^\infty (B_1)$, to show that certain boundary terms go to zero as $r\to \infty$. Therefore, the specific dependence of the constants is not relevant as long as they are finite. Instead, for the terms that remain through the estimates, it is important to have dependency only on the $L^1$ norm of the solution ---since weak solutions are only assumed to be in $L^1(B_1)$, and since for problem \eqref{Eq:ProblemLambda} the $L^1$ norm of $u_\lambda$, with $\lambda < \lambda^*$, is bounded uniformly in $\lambda$, as explained next.

\begin{remark}
	\label{Remark:UniformEstimatesInLambda}
	When one considers stable solutions $u_\lambda$ of \eqref{Eq:ProblemLambda} in general domains $\Omega$, the only available estimate that is uniform in $\lambda$ is the following:
	$$
	\norm{u_\lambda}_{L^1(\Omega)} \leq \norm{u^*}_{L^1(\Omega)} \quad \textmd{for all }\lambda < \lambda^*\,.
	$$
	Indeed, a simple argument shows that $\norm{u_\lambda}_{L^1(\Omega)}$ is uniformly bounded as $\lambda \uparrow \lambda^*$. Then, it follows that $u^*$ is a weak solution of \eqref{Eq:ProblemLambda}, i.e., belonging to $L^1(\Omega)$ (see \cite{RosOtonSerra-Extremal} for the details). In the case $\Omega = B_1$, the solutions $u_\lambda$ are radially decreasing (see Section~\ref{Sec:RadialSymmetry}). Hence, the $L^\infty$ norm of $u_\lambda$ in sets that are away from the origin is also bounded independently of $\lambda$, since in those sets it can be controlled by the $L^1$ norm of $u^*$. We have indeed
	$$
	\norm{u_\lambda}_{L^\infty (B_1\setminus \overline{B_R})} \leq \dfrac{C}{R^n} \norm{u_\lambda}_{L^1(B_1)} \leq \dfrac{C}{R^n} \norm{u^*}_{L^1(B_1)}  \quad \textmd{for every } R \in (0,1) \text{ and  } \lambda < \lambda^*\,.$$
	
	In fact, if $u\in L^1(B_1)$ is a weak solution of \eqref{Eq:SemilinearProblemB1} that is radially decreasing, automatically $u\in L^\infty_{\mathrm{loc}}(B_1\setminus\{0\})$. Then, by interior estimates for the fractional Laplacian (see Corollaries~2.3 and 2.5 in \cite{RosOtonSerra-Regularity}), $u$ is, in $B_1\setminus\{0\}$, at least as regular as the nonlinearity $f$. Since in this paper we assume $f\in C^2$, then we have $u\in C^{2,\alpha}_{\mathrm{loc}} (B_1\setminus\{0\})$ for some $\alpha > 0$. The hypothesis on $f$ can be slightly weakened depending on $s$, as it is explained in Remark~\ref{Remark:SharpfRegularity}.
\end{remark}

The following is the first result of this section (recall that we use the notation $r = |(x,y)|$).
\begin{lemma}
	\label{Lemma:Estimate for v_i and v_y}
	
	Let $u\in C^1 (B_1)\cap L^\infty(B_1)$ be such that $u\equiv 0$ in $\R^n \setminus B_1$, and let $v$ be its $s$-harmonic extension. Then, we have the following estimates:
	\begin{equation}
	\label{Eq:EstimateForv_i}	
	|v_{x_i} (x,y) | \leq C \dfrac{y^{2s}}{r^{n+1+2s}} \quad \textrm{ for } |x|>2,\ y>0\,,
	\end{equation}
	for $i=1,\ldots,n$, and
	\label{Lemma:EstimateForv_y}
	\begin{equation}
	\label{Eq:EstimateForv_y}	
	|v_y (x,y) | \leq C \dfrac{y^{2s - 1}}{r^{n+2s}} \quad \textrm{ for } |x|>2,\ y>0\,,
	\end{equation}
	for some constants $C$ depending only on $n$, $s$ and $\norm{u}_{L^\infty (B_1)}$.
	
\end{lemma}

The second result of this section also deals with the decay of $\nabla v$ as $y\to + \infty$. The main difference with the previous one is that we establish an estimate that does not depend on the $L^\infty$ norm of the solution, only on its $L^1$ norm. Therefore, it holds not only for bounded solutions but also for weak solutions ---recall \eqref{Eq:DefinitionWeakSolution}. As we will see, the result follows from an argument in the proof of Proposition~4.6 in \cite{CabreSireI}, and is the following.

\begin{proposition}
\label{Prop:GradientEstimate}
Let $s \in (0,1)$ and $a = 1-2s$. Let $v\in L^{2}_{\mathrm{loc}}(\R^{n+1}_+, y^a)$ satisfy $\nabla v \in L^{2}(\R^{n+1}_+, y^a)$  and $\div(y^a \nabla v) =0$ in  $\R_+^{n+1}$. Let $u$ be its trace on $\R^n$. Then, 
\begin{equation}
\label{Eq:GradientEstimate}
|\nabla v(x,y)| \leq \dfrac{C}{y^{n+1}} ||u||_{L^{1}(\R^n)} \ \ \textrm{ for every } \ y > 0
\end{equation}
and for a constant $C$ depending only on $n$ and $s$.
\end{proposition}

The third result we present is new and important. It provides an estimate for the horizontal gradient in the set $(B_{3/4}\setminus \overline{B_{1/2}}) \times (0,1)$. As it is commented in Remark~\ref{Remark:EstimateValidForBothOperators}, this gradient estimate is also valid for the problem studied in \cite{CapellaEtAl} for the operator $A^s$. Therefore, it can be used in the arguments of \cite{CapellaEtAl} in order to complete their proofs at the points where an estimate of this kind is missing (see Remarks~\ref{Remark:MissingTermCapellaEtAl} and \ref{Remark:MissingEstimateCapellaEtAl}).

\begin{proposition}
\label{Prop:HorizontalGradientEstimateRing}
Let $s\in (0,1)$ and $u \in H^s(\R^n)$ be a radially decreasing weak solution of \eqref{Eq:SemilinearProblemB1}, with $f\in C^{2}$. Let $v$ be the $s$-harmonic extension of $u$ given by \eqref{Eq:ExtensionProblemDirichlet} and
$$
A := \bpar{B_{3/4}\setminus \overline{B_{1/2}}} \times (0,1) \subset \R_+^{n+1}\,.
$$

Then, 
\begin{equation}
\label{Eq:HorizontalGradientEstimateHalfBalls}
\norm{\nabla_x v}_{L^{\infty}(A)} \leq C
\end{equation}
for some constant $C$ depending only on $n$, $s$, $\norm{u}_{L^1(B_1)}$,  $\norm{u}_{L^{\infty}(B_{7/8}\setminus B_{3/8})}$, and \\ $\norm{f'(u)}_{L^{\infty}(B_{7/8}\setminus B_{3/8})}$.
\end{proposition}

The rest of this section is devoted to prove Lemma~\ref{Lemma:Estimate for v_i and v_y}, Proposition~\ref{Prop:GradientEstimate} and Proposition~\ref{Prop:HorizontalGradientEstimateRing}. We start with the proof of the first lemma, which only relies on the Poisson formula for the $s$-harmonic extension of $u$.

\begin{proof}[Proof of Lemma~\ref{Lemma:Estimate for v_i and v_y}]
	Since $u$ has compact support in $\overline{B_1}$, by the Poisson formula we have
	$$
	v(x,y) = P * u = P_{n,s} \int_{B_1} \dfrac{y^{2s}}{\bpar{|x-z|^2 + y^2}^{\frac{n+2s}{2}}} u(z) \d z\,.
	$$
	If we differentiate the previous expression with respect to $x_i$, $i=1,\ldots,n$, we get
	$$
	|v_{x_i}| \leq C \norm{u}_{L^\infty (B_1)} y^{2s} \int_{B_1} \dfrac{|x_i - z_i|}{\bpar{|x-z|^2 + y^2}^{\frac{n+2+2s}{2}}} \d z \,.
	$$
	Now, on the one hand we use that $|x| > 2$ to see that 
	$$
	|x_i - z_i| \leq |x| + 1 \leq 2 |x| \leq 2r\,.
	$$
	On the other hand, 
	$$
	\dfrac{1}{\bpar{|x-z|^2 + y^2}^{\frac{n+2+2s}{2}}} \leq \dfrac{1}{\bpar{(|x|-1)^2 + y^2}^{\frac{n+2+2s}{2}}} \leq  \dfrac{4^{\frac{n+2+2s}{2}}}{\bpar{|x|^2 + y^2}^{\frac{n+2+2s}{2}}} = \dfrac{C}{r^{n + 2 + 2s}}\,,
	$$
	where in the first inequality we have used that $|x-z| \geq |x| - 1$ and in the second one, that $4|x-1|^2 \geq |x|^2$ if $|x| > 2$. Combining all this we get the  estimate \eqref{Eq:EstimateForv_i}.
	
	The proof for $v_y$ is completely analogous.
\end{proof}

We deal now with estimates for weak solutions. We start with the proof of Proposition~\ref{Prop:GradientEstimate}, establishing a gradient estimate for $v$ (the $s$-harmonic extension of $u$) in sets which are far from $y=0$.  To establish it we follow the ideas of Proposition~4.6 of \cite{CabreSireI}, but with a careful look on the right-hand side of the estimates.

\begin{proof}[Proof of Proposition~\ref{Prop:GradientEstimate}]
Let $(x_0, y_0) \in \R_+^{n+1}$ with $y_0 > 0$ and note that $v$ satisfies the equation $\div(y^a \nabla v) =0$ in $B_{y_0 / 2} (x_0, y_0)$. We perform the scaling $\overline{v}( \overline x , \overline y) = v(x_0 + y_0 \overline x, y_0 \overline y)$ and then $\overline{v}$ satisfies  $\div(\overline y^a \nabla \overline{v}) =0$ in $B_{1/ 2} (0, 1)$. Since $\overline y \in (1/2,3/2)$ in this ball, $\overline{v}$ satisfies a uniformly elliptic equation and we can use classical interior estimates for the gradient (see \cite{GilbargTrudinger}, Corollary 6.3) to obtain
$$
||\nabla \overline{v} ||_{L^{\infty}(B_{1/ 4}(0,1))} \leq C || \overline{v}  ||_{L^{\infty}(B_{1/2}(0,1))}\,,
$$
for a constant $C$ depending only on $n$.
Undoing the scaling we have
$$
|\nabla v (x_0, y_0)| \leq \dfrac{1}{y_0}||\nabla \overline{v} ||_{L^{\infty}(B_{1/ 4}(0,1))} \leq \dfrac{C}{y_0} || \overline{v}  ||_{L^{\infty}(B_{1/2}(0,1))} = \dfrac{C}{y_0} || v ||_{L^{\infty}(B_{y_0/2}(x_0, y_0))} \,.
$$

Finally, we estimate $|| v ||_{L^{\infty}(B_{y_0/2}(x_0, y_0))}$. Recall that $v = P * u$ and we can bound $P(x,y)$ by $P_{n,s}/y^n$ for every $y > 0$. Then,
$$
|v(x,y)| \leq \! \int_{\R^n} \! P(x-z, y)|u(z)| \d z \leq \dfrac{P_{n,s}}{y^n}  \! \int_{\R^n} \! |u(z)| \d z = \dfrac{P_{n,s}}{y^n} ||u||_{L^{1}(\R^n)} \ \ \textrm{ for every } y > 0 \,.
$$
Combining this with the previous estimate, we get \eqref{Eq:GradientEstimate}.
\end{proof}

The estimate given by Proposition~\ref{Prop:GradientEstimate} is useful to bound quantities far from $\{y=0\}$. However, in the proofs of Proposition~\ref{Prop:IntegrabilityHalfBall} and Theorem~\ref{Th:RadialMainTheorem} we also need to bound quantities up to $\{y=0\}$. This is done thanks to Proposition~\ref{Prop:HorizontalGradientEstimateRing}. To prove it we need two preliminary results, which are estimates in half-balls of $\R^{n+1}_+$. Regarding such sets, we use the notation
$$
B_{R}^+ = \setcond{(x,y)\in \R^{n+1}_+}{|(x,y)|< R},
$$
$$
\Gamma_{R}^0 = \setcond{(x,0)\in \partial \R^{n+1}_+}{|x|< R}.
$$
We also write $B_{R}^+(x_0)$ and $\Gamma_{R}^0(x_0)$ in order to denote that the center of the balls is $(x_0,0)$ and not the origin.

The first lemma we need is the following. It is used to bound the $L^\infty$ norm of $v$ in a half-ball $B_{R}^+$ by some quantities that only refer to the trace of $v$ on $\R^n$, $u$.

\begin{lemma}
\label{Lemma:BoundLInfinityExtension}
Let $s\in (0,1)$ and $u \in L^1(\R^n) \cap L^{\infty}_{\mathrm{loc}}(\R^n)$. Let $v$ be the $s$-harmonic extension of $u$ given by \eqref{Eq:ExtensionProblemDirichlet}. Then,
\begin{equation}
\label{Eq:BoundLInfinityExtension}
\norm{v}_{L^{\infty}(B_{R}^+)} \leq C \bpar{\norm{u}_{L^{\infty}(\Gamma_{2R}^0)} + \norm{u}_{L^{1}(\R^n)}}\,,
\end{equation}
where $C$ is a constant depending only on $n$, $s$ and $R$.
\end{lemma}

\begin{proof}
Let $(x,y) \in B_{R}^+$. By the Poisson formula,
$$
v(x,y) = \int_{\R^n} P(x-z, y) u(z) \d z\,, \quad \textrm{ where } \quad P(x, y) = P_{n,s} \dfrac{y^{2s}}{\bpar{|x|^2 + y^2}^{\frac{n+2s}{2}}}\,.
$$
Now, we split the integral into two parts:
$$
 \int_{\R^n} P(x-z, y) u(z) \d z =  \int_{\Gamma_{2R}^0} P(x-z, y) u(z) \d z +  \int_{\R^n \setminus \Gamma_{2R}^0} P(x-z, y) u(z) \d z\,.
$$
For the first term we find the estimate
$$
\int_{\Gamma_{2R}^0} P(x-z, y) u(z) \d z \leq  \norm{u}_{L^{\infty}(\Gamma_{2R}^0)} \int_{\R^n} P(x-z, y) \d z = \norm{u}_{L^{\infty}(\Gamma_{2R}^0)}\,,
$$
where we have used that $P(x,y)$ is positive and for all $y>0$ it integrates $1$ in $\R^n$.
For the second term, note that since $|x| < R$ and $|z| \geq 2R $, $|x-z| \geq R$ and therefore
$$
\dfrac{y^{2s}}{\bpar{|x-z|^2 + y^2}^{\frac{n+2s}{2}}} \leq \dfrac{y^{2s}}{\bpar{R^2 + y^2}^{\frac{n+2s}{2}}}\,.
$$
Hence, for $(x,y) \in B_{R}^+$,
$$
\int_{\R^n \setminus \Gamma_{2R}^0} P(x-z, y) u(z) \d z \leq
\int_{\R^n \setminus \Gamma_{2R}^0} P_{n,s}\dfrac{y^{2s}}{\bpar{R^2 + y^2}^{\frac{n+2s}{2}}} u(z) \d z \leq C \norm{u}_{L^{1}(\R^n)}\,,
$$
where $C$ is a constant depending only on $n$, $s$ and $R$.
\end{proof}

The second lemma we need in order to prove Proposition~\ref{Prop:HorizontalGradientEstimateRing} is a Harnack inequality:

\begin{lemma}[Lemma 4.9 of \cite{CabreSireI}]
\label{Lemma:HarnackInequalityNeumann}
Let $a\in (-1,1)$ and $\varphi \in H^{1}(B_{4R}^+, y^a)$ be a nonnegative weak solution of
$$
\beqc{\PDEsystem} 
\div(y^a \nabla \varphi) &=& 0 & \textrm{ in }  B_{4R}^+ \,,\\
\dfrac{\partial \varphi}{\partial \nu^a} + d(x)\varphi &=& 0 & \textrm{ in }  \Gamma_{4R}^0 \,,
\eeqc
$$
where $d$ is a bounded function in $\Gamma_{4R}^0$. Then, 
\begin{equation}
\label{Eq:HarnackInequalityNeumann}
\sup_{B_R^+}  \varphi \leq C \inf_{B_R^+}  \varphi \,,
\end{equation}
for some constant $C$ depending only on $n$, $a$ and $R^{1-a} \norm{d}_{L^{\infty}(\Gamma_{4R}^0)}$.
\end{lemma}

\begin{remark}
	Since the operator $\div (y^a \nabla \cdot)$ is invariant under translations in the $x$ variable, the two previous results also hold for half-balls not necessarily centered at the origin.
\end{remark}

Once we have the two previous lemmas, we can establish  Proposition~\ref{Prop:HorizontalGradientEstimateRing}:

\begin{proof} [Proof of  Proposition~\ref{Prop:HorizontalGradientEstimateRing}]
	
	We first claim that, for $x_0\in \setcond{x\in \R^n}{1/2  \leq |x| \leq 3/4}$, we have
	\begin{equation}
	\label{Eq:GradientEstimateHalfBalls}
	\norm{\nabla_x v}_{L^\infty (B^+_{1/32}(x_0))} \leq C\,,
	\end{equation}
	with $C$ depending only on $n$, $s$, $\norm{u}_{L^1(B_1)}$,  $\norm{u}_{L^{\infty}(\Gamma^0_{1/8}(x_0))}$ and $\norm{f'(u)}_{L^{\infty}(\Gamma^0_{1/8}(x_0))}$. Assuming that the claim is true we complete the proof. First, we use a standard covering argument to deduce
	$$
	\norm{\nabla_x v}_{L^\infty ((\Gamma^0_{3/4}\setminus \overline{ \Gamma^0_{1/2}}) \times (0,1/32))} \leq C\,,
	$$
	with a constant $C$ depending on the same quantities as the previous one. Then, we use Proposition~\ref{Prop:GradientEstimate} to bound $\nabla_x v$ in $(\Gamma^0_{3/4}\setminus \overline{ \Gamma^0_{1/2}}) \times (1/32,1)$. Combining these last two estimates we deduce \eqref{Eq:HorizontalGradientEstimateHalfBalls}.
	
	Let us show \eqref{Eq:GradientEstimateHalfBalls}. By the radial symmetry of the domain, it is enough to prove the estimate for a point $x_0$ of the form $x_0 = (c, c, \ldots , c)$ with $c$ such that $1/2 \leq |x_0| \leq 3/4$. Under these assumptions, the ball $\Gamma^0_{1/8}(x_0)$ is inside the first orthant of $\R^n$, i.e., $\Gamma^0_{1/8}(x_0) \subset \{x_i \geq 0,\ i = 1,\ldots, n\}$, and there we have $u_{x_i} < 0$ for all $i = 1,\ldots,n$ (and the same happens for $v_{x_i}$).
	Since the equation that $v$ satisfies is invariant under translations in the $x$ variable, we can assume from now on that the ball is centered at the origin, so we write just $B_{1/8}^+$.
	
	Then, differentiating the equation that $v$ satisfies in $B_{1/8}^+$, for all  $i = 1,\ldots,n$ we get 
	$$
	\beqc{\PDEsystem} 
	\div(y^a \nabla v_{x_i}) &=& 0 & \textrm{ in }  B_{1/8}^+ \,,\\
	\dfrac{\partial v_{x_i}}{\partial \nu^a}  -f'(u)v_{x_i} &=& 0 & \textrm{ in }  \Gamma_{1/8}^0 \,.
	\eeqc
	$$
	At this point we use Lemma~\ref{Lemma:HarnackInequalityNeumann} with $\varphi = -v_{x_i} \geq 0$ and $d = -f'(u)$, obtaining
	$$
	\sup_{B^+_{1/32}}  - v_{x_i} \leq C \inf_{B^+_{1/32}}  - v_{x_i}\,,
	$$
	with a constant $C$ depending only on $n$, $s$ and $\norm{f'(u)}_{L^{\infty}(\Gamma_{1/8}^0)}$.
	
	We bound $\inf_{B^+_{1/32}} - v_{x_i}$ by $C\norm{v}_{L^{\infty}(B_{1/32}^+)}$ with a constant $C$ depending only on~$n$. To see this, we use integration by parts:
	$$
	\inf_{B_{1/32}^+}  - v_{x_i} \leq \norm{v_{x_i}}_{L^{1}(B_{1/32}^+)} = \left | \int_{B_{1/32}^+} v_{x_i} \d x \right | = \left | \int_{\partial (B_{1/32}^+)} v \nu_i \d \sigma \right | \leq C\norm{v}_{L^{\infty}(B_{1/32}^+)}. 
	$$
	
	Finally, \eqref{Eq:GradientEstimateHalfBalls} is obtained by estimating $\norm{v}_{L^{\infty}(B_{1/32}^+)}$ in terms of $u$, the trace of $v$ on $\R^n$, using Lemma~\ref{Lemma:BoundLInfinityExtension}.
\end{proof}

Obviously, in the definition of the set $A$ of Proposition~\ref{Prop:HorizontalGradientEstimateRing} we can replace $B_{3/4} \setminus \overline{ B_{1/2}}$ by every other annuli  $B_L\setminus \overline{B_l}$ with $0<l<L<1$, and we can also replace $(0,1)$ by any other open interval. Then, estimate \eqref{Eq:HorizontalGradientEstimateHalfBalls} also holds for $A = (B_L\setminus \overline{B_l}) \times (0,T)$ with a different constant $C$ which depends on $l$ and $L$.

\begin{remark}
\label{Remark:EstimateValidForBothOperators}
Let $u\geq 0$ be a bounded solution of \eqref{Eq:SemilinearProblemB1} and let $w\geq0$ be a bounded solution of the problem
$$
\beqc{\PDEsystem}
    A^s w &= & f(w) & \text{in} \quad B_1\,,\\
    w &= &0 &   \text{on} \quad \partial B_1\,.
\eeqc
$$
Then, in the half-ball $B^+_{1/8}$ (or in every half-ball with base strictly contained in $B_1$), both $u$ and $w$ satisfy the same degenerate elliptic problem. Therefore, the proof of Proposition~\ref{Prop:HorizontalGradientEstimateRing} can be applied without any change to $w$. Thus, we obtain an estimate for $\nabla_x w$ that can be used in the arguments of \cite{CapellaEtAl} in order to complete the proof of their main theorem (see Remarks~\ref{Remark:MissingTermCapellaEtAl} and \ref{Remark:MissingEstimateCapellaEtAl}).
\end{remark}

\section{Radial symmetry and monotonicity of stable solutions}
\label{Sec:RadialSymmetry}

In this section we establish the radial symmetry of bounded stable solutions and that, when they are not identically zero, they are either increasing or decreasing.

As it is well known, when $u\geq 0$ is a bounded solution of \eqref{Eq:SemilinearProblemB1}, then $u$ is radially symmetric and decreasing ($u_\rho < 0$ for $1>\rho > 0$). This was proved in \cite{Birkner2005} using the celebrated moving planes method.  Furthermore, by the Poisson formula, the $s$-harmonic extension of $u$ is also radially symmetric in the horizontal direction, that is, it only depends on $\rho$ and $y$. Moreover, $v_\rho < 0$ for $\rho > 0$.

In the moving planes argument, the hypothesis of $u\geq 0$ cannot be omitted, since there can be changing-sign solutions of \eqref{Eq:SemilinearProblemB1} that are not radially symmetric. Nevertheless, this is not the case for stable solutions, as the next result states:

\begin{proposition}
	\label{Prop:RadialSymmetryAndMonotonicity}
	Let $n\geq 2$ and let $u$ be a bounded stable solution of \eqref{Eq:SemilinearProblemB1} with $f\in C^2$. Then, $u$ is radially symmetric. Moreover, if $u$ is not identically zero then $u$ is either increasing or decreasing in $B_1\setminus\{0\}$.	
\end{proposition}

The first part of this result is already well known (see for instance Remark~5.3 of~\cite{RosOtonSerra-Extremal}), but we will present here the proof for completeness.  Instead, to our knowledge, the second part of the proposition about the monotonicity has not been established in the nonlocal setting. In order to prove it, we follow the main ideas in the classical proof of the analogous result for the Laplacian ($s=1$), which can be found for instance in \cite{CabreCapella-Radial, Dupaigne}. 
The argument in the local case is quite simple: one must show that if $u_\rho$ is not identically zero in $B_1$, then it cannot vanish in $B_1\setminus \{0\}$. As a consequence, either $u_\rho > 0$ or $u_\rho < 0$ in $B_1\setminus \{0\}$. Hence, to complete the proof, we assume that there exists $\rho_\star \in (0,1)$ for which $u_\rho(\rho_\star) = 0$ and $u_\rho \not \equiv 0$ in $\omega := B_{\rho_\star}$. Therefore, $u_\rho \chi_\omega \in H^1_0(B_1)$ and we can take it as a test function in the stability condition. Finally, we get the contradiction after an integration by parts in $\omega$.

Adapting the previous argument to the nonlocal case using the extension problem is not a straightforward task. To do it, we choose $v_\rho \chi_{\Omega}$ as a test function in \eqref{Eq:StabilityConditionExtension} to arrive at a contradiction. Here, $v$ is the $s$-harmonic extension of $u$ and $\Omega \subseteq \R_+^{n+1}$ is a certain connected component of the set $\{v_\rho \neq 0\}$ that must be chosen appropriately to satisfy the following condition. We need that $\overline{\partial \Omega \cap B_1} \cap \partial B_1 = \emptyset$, since this condition guarantees that $u\in C^2(\overline{\partial \Omega \cap B_1})$, a property that will be used in our arguments. Note that $u$ is not $C^2$ in a neighborhood of $\partial B_1$. Recall ---see \cite{RosOtonSerra-Regularity}--- that $u \sim \delta^s$ near $\partial B_1$, where $\delta = \dist (\cdot, \partial B_1)$. In particular, $u_\rho \notin L^2(B_1)$ for $s \leq 1/2$.
As a consequence of this, $\partial \Omega \cap B_1$ may differ from $B_{\rho_\star}$ (where $u_\rho (\rho_\star) = 0$) in contrast with the local case.

In addition, $\Omega$ may turn to be unbounded. For this reason we need Lemma~\ref{Lemma:Estimate for v_i and v_y} and Proposition~\ref{Prop:GradientEstimate} to control the decay at infinity of $\nabla v$. This is necessary in order to perform correctly an integration by parts in $\Omega$. 

We proceed now with the detailed proof.

\begin{proof}[Proof of Proposition~\ref{Prop:RadialSymmetryAndMonotonicity}]

	We first show the symmetry of $u$, following \cite{RosOtonSerra-Extremal}. For $i\neq j$ and $i,j = 1,\ldots,n$, consider $w = x_i u_{x_j} - x_j u_{x_i}$, which is a function defined in $\R^n$. Define its extension in $\R^{n+1}_+$ as $W = x_i v_{x_j} - x_j v_{x_i}$, where $v$ is the $s$-harmonic extension of $u$.
	Then,
	\begin{align*}
	\div(y^a \nabla W) &= y^a \Delta_x W + \partial_y (y^a W_y)\\
	&= y^a (x_i \Delta_x v_{x_j} - x_j \Delta_x v_{x_i} ) + \partial_y (y^a x_i (v_y)_{x_j} - y^a x_j (v_y)_{x_i}) \\
	&= x_i (\div(y^a \nabla v_{x_j})) - x_j (\div(y^a \nabla v_{x_i})) \\
	&=0 
	\end{align*}
	and
	\begin{align*}
	d_s \dfrac{\partial W}{\partial \nu^a} 
	&= x_i d_s \dfrac{\partial v_{x_j}}{\partial \nu^a}  - x_j d_s \dfrac{\partial v_{x_i}}{\partial \nu^a} \\
	&=  x_i f'(u) u_{x_j} - x_j f'(u) u_{x_i} \\
	&= f'(u) w\,.
	\end{align*}
	This means that $W$ is a solution of the linearized problem
	$$
	\beqc{\PDEsystem}
	\div (y^a \nabla W ) &= & 0 & \text{in } \R^{n+1}_+ \,,\\
	d_s \dfrac{\partial W}{\partial \nu^a}  &= & f'(u) w &  \text{in }  B_1\subset \R^n\,.
	\eeqc
	$$
	Equivalently, the trace of $W$ on $\R^n$, $w$, solves
	$$
	\beqc{\PDEsystem}
	\fraclaplacian w  &= & f'(u)w & \text{in } B_1 \,,\\
	w &= & 0 &  \text{in }  \R^n \setminus B_1\,.
	\eeqc
	$$
	
	Let us prove that $w\equiv 0$ for every $i\neq j$, $i,j = 1,\ldots,n$. This leads to the radial symmetry of $u$ since all its tangential derivatives are zero.
	
	Due to the stability of $u$, we have that $\lambda_1 (\fraclaplacian - f'(u) ; B_1) \geq 0$, that is, the first eigenvalue of the operator $\fraclaplacian - f'(u)$ in $B_1$ with zero Dirichlet data outside $B_1$ is nonnegative. Here we have to consider two cases. If $\lambda_1 (\fraclaplacian - f'(u) ; B_1) > 0$, then $w \equiv 0$. On the contrary, if $\lambda_1 (\fraclaplacian - f'(u) ; B_1) = 0$ then $w=K\phi_1$, that is, $w$ is a multiple of the first eigenfunction $\phi_1$, which is positive. But since $w$ is a tangential derivative, it cannot have constant sign along a sphere $\{|x| = R\}$ for $R\in(0,1)$. Hence, $K=0$, which leads to $w\equiv 0$. Thus, $u$ is radially symmetric.
	
	We prove now the second part of the result. In order to establish the monotonicity of $u$, it is enough to see that if $u_\rho \not \equiv 0$ in $B_1$, then $u_\rho$ does not vanish in $B_1\setminus \{0\}$. If this is shown to be true, then either $u_\rho > 0$ or $u_\rho < 0$ in $B_1\setminus \{0\}$. 
	
	Arguing by contradiction, we assume that there exists $\rho_\star \in (0,1)$ such that $u_\rho (\rho_\star) = 0$.	Let
	$$
	A^+ = \{v_\rho > 0\} \quad \textrm{ and } A^- = \{v_\rho < 0\}\,.
	$$
	Assume first that one of these two open sets is empty, for instance $A^- = \emptyset$ (the other case is analogous). Then, we find a contradiction with Hopf's lemma. Indeed, since $A^- = \emptyset$, $v_\rho$ satisfies
	$$
	\beqc{\PDEsystem}
	\div(y^a \nabla v_\rho) &=& y^a \dfrac{n-1}{\rho^2} v_\rho & \textrm{ in } \R_+^{n+1}\,, \\
	v_\rho &\geq & 0 & \textrm{ in } \R_+^{n+1}\,, \\
	\dfrac{\partial v_\rho}{\partial \nu^a} &=& f'(u) u_\rho& \textrm{ in } B_1\,.
	\eeqc
	$$
	At the same time, $v_\rho (\rho_\star,0) = u_\rho(\rho_\star) = 0$ and thus
	$$\dfrac{\partial v_\rho}{\partial \nu^a} (\rho_\star, 0) = f'(u(\rho_\star)) u_\rho(\rho_\star) = 0\,.
	$$
	This contradicts the Hopf's lemma for the operator $\tilde{L}_a w := \div(y^a \nabla w) - y^a c(x)w$, with $c = (n-1)/\rho^2$, which can be proved with the same arguments as in Proposition~4.11 of~\cite{CabreSireI}.

	Assume now that $A^+ \neq \emptyset$  and $A^- \neq \emptyset$. Our goal is to get a contradiction with the stability of $u$. For this, we need to define a set $\Omega \subset \R_+^{n+1}$ for which $v_\rho \chi_{\Omega} \in H^1(\R_+^{n+1}, y^a)$ ---note that this forces $v_\rho \equiv 0$ on $\partial \Omega \cap \{y>0\}$--- and, thus, $v_\rho \chi_{\Omega}$ is a valid test function in the stability condition. The resulting relation must then be integrated by parts in $\Omega$. This will require the integral 
	$$
	\int_{\partial \Omega \cap B_1} f'(u)u_\rho^2 \zeta_\varepsilon \d x\,, 
	$$
	to be finite, where $\zeta_\varepsilon$ is a smooth function. Now, since $u_\rho \notin L^2(B_1)$ for $s \leq 1/2$, we need to choose $\Omega$ such that $\overline{\partial\Omega \cap B_1} \cap \partial B_1$ is empty and, therefore, $u\in C^2(\overline{\partial\Omega \cap B_1})$.
	
	To accomplish this, we first make the following
	
	\textbf{Claim 1: } There exists a set $\Omega \subset \R_+^{n+1}$ (perhaps unbounded) such that $v_\rho$ does not vanish in $\Omega$, $v_\rho = 0$ on $\partial \Omega \cap \{y > 0\}$ and such that
	$$
	\overline{\partial \Omega \cap B_1} \cap \partial B_1 = \emptyset\,.
	$$
	
	To show this, we define
	$$
	A^+_0 = \partial A^+ \cap B_1 \quad \textrm{ and }A^-_0 = \partial A^- \cap B_1\,.
	$$ 	
	Note that if $u_\rho \leq 0$ in $B_1$, the Poisson formula yields $v_\rho \leq 0$ in $\R_+^{n+1}$. Similarly, $u_\rho \geq 0$ in $B_1$ ensures that $v_\rho \geq 0$ in $\R_+^{n+1}$. Therefore, since  $A^+ \neq \emptyset$ and $A^- \neq \emptyset$, we also have $A^+_0 \neq \emptyset$  and $A^-_0 \neq \emptyset$. 
	
	Since $v$ is radially symmetric in the horizontal variables, we can identify the sets $A^+$, $A^-$, $A^+_0$ and $A^-_0$ with their projections into $\R^2_{++} := \{(\rho,y) \in \R^2 : \rho,y \geq 0\}$ and recover the original sets by a revolution about the $y$-axis. With this identification in mind, let $(\rho_-,0) \in A^-_0$ and $(\rho_+,0) \in A^+_0$. Without loss of generality, we can assume that $\rho_- < \rho_+$ ---the argument in the other case is analogous. Let $\Omega_-$ be the connected component of $A^-$ whose closure contains $(\rho_-,0)$, and let $\Omega_+$ be the connected component of $A^+$ whose closure contains $(\rho_+,0)$. Now, we distinguish two cases.

	\textbf{Case 1: $\overline{\partial\Omega_- \cap B_1} \cap \partial B_1 = \emptyset $.}
	
	In this case we define
	$$
	\Omega:= \Omega_-\,.
	$$

	\textbf{Case 2:  $\overline{\partial\Omega_- \cap B_1} \cap \partial B_1 \neq \emptyset $.}
	
	In this case, a simple topological argument yields that  $\overline{\partial\Omega_+ \cap B_1} \cap \partial B_1 = \emptyset$. Indeed, under the assumption of Case~2, there exists $(\rho_-',0) \in \partial\Omega_- \cap B_1$ as close as we want to $\partial B_1$ and such that $\rho_- < \rho_+ < \rho_-' < 1$. Since $\Omega_-$ is arc-connected, we can join $(\rho_-,0)$ and $(\rho_-',0)$ by a curve in $\Omega_- \cap \{y > 0\}$. By the Jordan curve theorem, the connected component $\Omega_+$, whose closure contains $(\rho_+,0)$, is bounded and satisfies  $\overline{\partial\Omega_+ \cap B_1} \cap \partial B_1 = \emptyset$.
	
	Thus, in Case~2 we define
	$$
	\Omega:= \Omega_+
	$$
	and Claim~1 is proved.
	
	To proceed, me make the following
	
	\textbf{Claim~2: } $v_\rho \chi_\Omega \in H^1(\R^{n+1}_+, y^a)$ and the following formula holds:
	\begin{equation}
	\label{Eq:IntegrationByParts}
	(n-1) d_s \int_{\Omega} y^a \dfrac{v_\rho^2}{\rho^2} \d x \d y  = - d_s \int_{\Omega} y^a |\nabla v_\rho|^2 \d x \d y  + \int_{B_1 \cap \partial \Omega} f'(u)u_\rho^2  \d x\,.
	\end{equation}
	
	To prove Claim~2, note first that $v_\rho$ satisfies the equation
	$$
	\div(y^a\nabla v_\rho) = y^a \dfrac{n-1}{\rho^2}v_\rho \quad \textrm{ in } \ \R^{n+1}_+\,.
	$$
	Take $\zeta_\varepsilon = \zeta_\varepsilon(\rho)$ a smooth cut-off function such that $\zeta_\varepsilon = 0$ in $B_\varepsilon$ and $\zeta_\varepsilon = 1$ outside $B_{2\varepsilon}$. Multiply the above equation by $d_s v_\rho(\rho,y)  \zeta_\varepsilon(\rho) \chi_{\Omega} (\rho,y)$ and integrate in $\R^{n+1}_+$. Using integration by parts and the fact that $u_\rho = 0$ in $\R^n \setminus B_1$, we get
	\begin{align}
	\label{Eq:IntegrationByPartsDetailed}
	(n-1) d_s \int_{\Omega} y^a \dfrac{v_\rho^2}{\rho^2}\zeta_\varepsilon  \d x \d y & = d_s \int_{\Omega} \div(y^a\nabla v_\rho) v_\rho \zeta_\varepsilon \d x \d y \nonumber  \\
	&= - d_s \int_{\Omega} y^a\nabla v_\rho \cdot \nabla ( v_\rho \zeta_\varepsilon ) \d x \d y  + \int_{B_1 \cap \partial \Omega} f'(u)u_\rho^2 \zeta_\varepsilon \d x\,.
	\end{align}
	
	At this point, we need to justify this integration by parts. On the one hand, we know that $v_\rho = 0$ on $\partial \Omega \cap \{y>0\}$, and therefore there are no boundary terms except for the one in $\partial \Omega \cap B_1$. Note that, since $u\in C^2(\overline{B_1 \cap \partial \Omega})$, we have
	\begin{equation}
	\label{Eq:IntegralPartialOmegaFinite}
	\int_{B_1 \cap \partial \Omega} f'(u)u_\rho^2 \zeta_\varepsilon \d x <  +\infty\,.
	\end{equation}
	On the other hand, since $\Omega$ may be unbounded, the right way to do the computation in \eqref{Eq:IntegrationByPartsDetailed} is the following: we first integrate by parts in half-balls $B_R^+$ and then we make $R\to \infty$. We need to ensure that the boundary terms in $\{y>0\}$ go to zero, i.e.,
	\begin{equation}
	\label{Eq:BoundaryTermsGoToZero}
	\int_{\partial B_R^+ \cap \{y>0\} \cap \Omega} y^a v_\rho \dfrac{\partial v_\rho}{\partial \nu} \zeta_\varepsilon  \d \sigma \to 0 \quad \text{ as } R\to +\infty\,.
	\end{equation}
	This can be easily seen by using the estimate of Lemma~\ref{Lemma:Estimate for v_i and v_y} at the points with $|x|>2$. For the other points, by Proposition~\ref{Prop:GradientEstimate} we have
	$$
	|\nabla v (x,y) | \leq  \dfrac{C}{y^{n+1}} =  \dfrac{C}{(R^2 - |x|^2)^{\frac{n+1}{2}}}\,,
	$$
	for a constant $C$ depending only on $n$, $s$ and $\norm{u}_{L^1 (B_1)}$. Here we have used that $|x|^2 + y^2 = R^2$. Then, we take into account that
	$$
	\dfrac{1}{R^2 - |x|^2} \leq \dfrac{2}{R^2} \quad \textrm{ if } R > 2\sqrt{2} \ \textrm{ and }\ |x| < 2 
	$$
	to deduce 
	\begin{equation}
	\label{Eq:DecayEstimateRadialSymetryProof}
		|\nabla v (x,y) | \leq \dfrac{C}{R^{n+1}} \quad  \textrm{ if } R > 2\sqrt{2}\ \textrm{ and }\ |x| < 2 \,.
	\end{equation}
	Combining this estimate with the ones in Lemma~\ref{Lemma:Estimate for v_i and v_y}, we deduce \eqref{Eq:BoundaryTermsGoToZero}.
	
	In addition, by Lemma~\ref{Lemma:Estimate for v_i and v_y} and \eqref{Eq:DecayEstimateRadialSymetryProof}, the left-hand side of \eqref{Eq:IntegrationByPartsDetailed} is finite for all $\varepsilon \in [0,1/2]$. Hence all the quantities appearing in \eqref{Eq:IntegrationByPartsDetailed} are finite ---recall \eqref{Eq:IntegralPartialOmegaFinite}--- and, letting $\varepsilon \to 0$, we deduce \eqref{Eq:IntegrationByParts}. Furthermore,
	$$
	\int_{\Omega} y^a |\nabla v_\rho |^2\d x \d y <+\infty\,.
	$$
	This and the fact that $v_\rho = 0$ on $\partial \Omega \cap \{y>0\}$ yield that $v_\rho \chi_{\Omega}\in H^1(\R^{n+1}_+, y^a)$. Therefore, Claim~2 is proved.
	
	We conclude now the proof. Since $v_\rho \chi_\Omega \in H^1(\R^{n+1}_+, y^a)$, we can take it in the stability condition \eqref{Eq:StabilityConditionExtension} to obtain
	\begin{equation}
	\label{Eq:Stabilityv_rho}
	0 \leq d_s \int_{\Omega} y^a |\nabla v_\rho |^2\d x \d y  - \int_{B_1 \cap \partial \Omega} f'(u)u_\rho^2 \d x\,.
	\end{equation}
	Combining this with \eqref{Eq:IntegrationByParts} and using that $n \geq 2 $ and $d_s > 0$, we get
	$$
	0 \leq d_s \int_{\Omega} y^a |\nabla v_\rho |^2\d x \d y  - \int_{B_1 \cap \partial \Omega} f'(u)u_\rho^2 \d x =  - (n-1) d_s \int_{\Omega} y^a \dfrac{v_\rho^2}{\rho^2} \d x \d y < 0\,,
	$$
	a contradiction. 	
\end{proof}

\section{Weighted integrability. Proof of Proposition~\ref{Prop:IntegrabilityHalfBall}}
\label{Sec:WeightedIntegrability}

This section is devoted to establish Proposition~\ref{Prop:IntegrabilityHalfBall}, which is the key ingredient in the proof of Theorem~\ref{Th:RadialMainTheorem}. To do so, we first need the following lemma, which is an expression of the stability condition when the test function $\xi$ is taken as $\xi = c \eta$, with $c$ to be chosen freely and $\eta$ with compact support. 

\begin{lemma}
\label{Lemma:StabilityCtimesEta}
Let $s\in (0,1)$ and $a = 1-2s$. Let $f$ be a nondecreasing $C^1$ function and $u$ a stable weak solution of
$$
\beqc{\PDEsystem}
    \fraclaplacian u &= & f(u) &  \text{in} \quad   \Omega\,,\\
    u &= &0 &  \text{in} \quad   \R^n\setminus \Omega \,.
\eeqc
$$ 
Let $v$ be the $s$-harmonic extension of $u$. 

Then, for all $c\in H^1_{\mathrm{loc}}(\R^{n+1}_+, y^a)$ and $\eta \in C^{1}(\overline{\R^{n+1}_+})$ with compact support and such that its trace has support in $\Omega \times \{ 0\}$,
$$
\int_\Omega \left \{ f'(u) c - d_s \dfrac{\partial c}{\partial \nu^a}\right \} c \eta^2 \d x \leq  
d_s\int_{\R^{n+1}_+} y^a c^2  |\nabla \eta |^2 \d x \d y -
d_s\int_{\R^{n+1}_+} \div (y^a \nabla c ) c \eta^2 \d x \d y ,
$$
where $d_s$ is the best constant of the trace inequality \eqref{Eq:TraceInequalityHalfSpace}.
\end{lemma}

\begin{proof}
Simply take $\xi = c\eta$ in the stability condition \eqref{Eq:StabilityConditionExtension} and integrate by parts:
\begin{align*}
\int_\Omega f'(u) c^2 \eta^2 \d x & \leq d_s   \int_{\R^{n+1}_+} y^a \left \{ c^2|\nabla \eta|^2 +  \eta^2|\nabla c|^2 + c \nabla c \cdot \nabla \eta^2 \right \}\d x \d y \\
& =    d_s \int_\Omega  \dfrac{\partial c}{\partial \nu^a} c \eta^2 \d x  +  d_s\int_{\R^{n+1}_+} y^a c^2  |\nabla \eta |^2 \d x \d y \\
& \qquad -
d_s\int_{\R^{n+1}_+} \div (y^a \nabla c ) c \eta^2 \d x \d y \,.
\end{align*}
\end{proof}

Thanks to this lemma we can now prove Proposition~\ref{Prop:IntegrabilityHalfBall}:

\begin{proof}[Proof of Proposition~\ref{Prop:IntegrabilityHalfBall}]

We first note that we can replace the conditions on $c$ and $\eta$ in Lemma~\ref{Lemma:StabilityCtimesEta} by the following: $c \in H^1_{\mathrm{loc}}(\R^{n+1}_+ \setminus \{0\}, y^a)$ and $\eta \in C^{1}(\R^{n+1}_+)$ with $\tr \eta \in C^{1}_0 (B_1\setminus \{0\})$, where $\tr$ denotes the trace on $\R^n$. Therefore, we can take $c = v_\rho$, which belongs to $ H^1_{\mathrm{loc}}(\R^{n+1}_+ \setminus \{0\}, y^a)$. To see this, recall that $u\in C^2_{\mathrm{loc}} (B_1\setminus\{0\})$ (see Remark~\ref{Remark:UniformEstimatesInLambda}). Hence, using the estimates given by Proposition~\ref{Prop:GradientEstimate} and Proposition~\ref{Prop:HorizontalGradientEstimateRing}, we deduce that $\nabla_x v \in L^\infty_{\mathrm{loc}}(\R^{n+1}_+ \setminus \{0\})$, which yields  $v_\rho \in H^1_{\mathrm{loc}}(\R^{n+1}_+ \setminus \{0\}, y^a)$.

Differentiating with respect to $\rho$ the equation $\div(y^a \nabla v) = 0$  and the boundary condition $d_s \partial_{\nu ^a}v = f(u)$, we  have the following equations for $c=v_\rho$:
$$
\div(y^a\nabla c) = \div(y^a\nabla v_\rho) = y^a \dfrac{n-1}{\rho^2}v_\rho \quad \textrm{ in } \R^{n+1}_+
$$
and
$$
d_s\dfrac{\partial c}{\partial \nu ^a}  =d_s\dfrac{\partial v_\rho}{\partial \nu ^a} = f'(u) u_\rho = f'(u) c \quad \textrm{ in } B_1\,.
$$
Therefore, we take $c=v_\rho $ in Lemma~\ref{Lemma:StabilityCtimesEta} to get
$$
    (n-1)\int_{\R^{n+1}_+} y^a \dfrac{(v_\rho \eta )^2}{\rho^2} \d x \d y \leq 
    \int_{\R^{n+1}_+} y^a v_\rho^2 |\nabla \eta |^2 \d x \d y\,,
$$
for every $\eta \in C^{1}(\R^{n+1}_+)$ with compact support and such that $\tr \eta \in C^{1}_0 (B_1\setminus \{0\})$. For the purpose of our computations, it is convenient to replace $\eta$ by $\rho \eta$, thus obtaining
\begin{equation}
\label{Eq:InequalityTestFunctions}
(n-1)\int_{\R^{n+1}_+} y^a v_\rho^2 \eta^2 \d x \d y \leq 
\int_{\R^{n+1}_+} y^a v_\rho^2 |\nabla (\rho \eta) |^2 \d x \d y\,.
\end{equation}

Now, we proceed with some cut-off arguments. Let $\zeta_{\delta} $ and  $\psi_T$ be two functions in $C^\infty(\R)$ such that 
$$
\zeta_{\delta}(\rho) =
\begin{cases}
0 & \textrm{ if }  \rho \leq \delta\,,\\
1 & \textrm{ if }  \rho \geq 2\delta\,,
\end{cases}
 \ \ \
 \zeta_{\delta}' (\rho) \leq \dfrac{C}{\delta} \textrm{ if } \rho\in (\delta, 2\delta) 
 \ \ \
 \textrm{ and }  
 \ \ \
\psi_T(y) =
\begin{cases}
1 & \textrm{ if }  y \leq T\,,\\
0 & \textrm{ if }  y \geq T+1\,.
\end{cases}
$$
Then, we take 
$$
\eta (\rho, y) = \eta_\varepsilon (\rho) \psi_T (y) \zeta_{\delta}(\rho)
$$
in \eqref{Eq:InequalityTestFunctions}, where $\eta_\varepsilon $ is a $C^1$ function with compact support in $B_1$ to be choosen later. We assume that $\eta_\varepsilon$ and $|\nabla (\rho  \eta_\varepsilon  )|$ are bounded in $B_1$. Therefore, we obtain
\begin{align}
\label{Eq:InequalityTestFunctions2}
\begin{split}
	(n-1) \int_0^{T} \int_{B_1} y^a v_\rho^2 \eta_\varepsilon^2  \zeta_{\delta}^2 \d x \d y 
	&\leq  
	\int_0^{T+1} \int_{B_1}  y^a v_\rho^2 |\nabla (\rho  \eta_\varepsilon \psi_T \zeta_{\delta}) |^2 \d x \d y \\
	& \leq 
	\int_0^{T+1} \int_{B_1}  y^a v_\rho^2 |\nabla (\rho  \eta_\varepsilon \zeta_{\delta} ) |^2 \psi_T^2 \d x \d y \\
	& \quad \quad + \int_T^{T+1} \int_{B_1}  y^a v_\rho^2  \eta_\varepsilon^2 |\nabla (\psi_T) |^2  \d x \d y\,.
\end{split}
\end{align}
Now, we see that
\begin{align}
\label{Eq:InequalityTestFunctions3}
\begin{split}
\int_0^{T+1} \! \! \int_{B_1}  y^a v_\rho^2 |\nabla (\rho  \eta_\varepsilon \zeta_{\delta} ) |^2 \psi_T^2 \d x \d y = \hspace{-58mm} \\
 &= 
 \int_0^{T+1} \! \! \int_{B_1}  y^a v_\rho^2 \psi_T^2 \left \{|\nabla (\rho  \eta_\varepsilon  ) |^2 \zeta_{\delta}^2 + \rho^2 \eta_\varepsilon^2 |\nabla \zeta_{\delta}|^2  + 2 \rho \eta_\varepsilon \zeta_{\delta} \nabla \zeta_{\delta} \cdot \nabla (\rho \eta_\varepsilon)  \right \}  \d x \d y \\
 &\leq 
 \int_0^{T+1} \! \! \int_{B_1}  y^a v_\rho^2 |\nabla (\rho  \eta_\varepsilon  ) |^2   \d x \d y 
 \\
 &\quad \quad + C  \int_0^{T+1} \! \!\int_{B_{2\delta} \setminus B_\delta}  y^a v_\rho^2  |\eta_\varepsilon| \left \{ \dfrac{\rho^2}{\delta^2} |\eta_\varepsilon|  +  \dfrac{\rho}{\delta}  \zeta_{\delta} |\nabla (\rho \eta_\varepsilon)|  \right \}  \d x \d y \\
 &\leq 
 \int_0^{T+1} \! \! \int_{B_1}  y^a v_\rho^2 |\nabla (\rho  \eta_\varepsilon  ) |^2   \d x \d y 
 +
 C \int_0^{T+1} \! \! \int_{B_{2\delta} \setminus B_\delta}  y^a v_\rho^2   \d x \d y \,.
\end{split}
\end{align}
Note that in the last inequality we have used that $\eta_\varepsilon$ and $|\nabla (\rho  \eta_\varepsilon  )|$ are bounded. Since $u\in H^s(\R^n)$, we have that its $s$-harmonic extension, $v$, is in $H^1(\R_+^{n+1}, y^a)$ (see the comments in Section~\ref{Sec:ExtensionProblem}). Therefore, the last term in the previous inequalities tends to zero as $\delta \to 0$. Exactly as in the local case (see the proof of Lemma~2.3 in \cite{CabreCapella-Radial}), this point is the only one where we use that $u\in H^s(\R^n)$. Hence, combining \eqref{Eq:InequalityTestFunctions2} and \eqref{Eq:InequalityTestFunctions3}, and letting $\delta \to 0$, by monotone convergence we have
\begin{align*}
(n-1) \int_0^{T} \int_{B_1} y^a v_\rho^2 \eta_\varepsilon^2  \d x \d y 
&\leq  
\int_0^{T} \int_{B_1}  y^a v_\rho^2 |\nabla (\rho  \eta_\varepsilon ) |^2 \d x \d y \\
& \quad \quad + \int_T^{T+1} \int_{B_1}  y^a v_\rho^2  \{  |\nabla (\rho  \eta_\varepsilon ) |^2 +  \eta_\varepsilon^2  |\nabla (\psi_T) |^2 \} \d x \d y\,.
\end{align*}
Now we want to make $T\to \infty$. We claim that the last term in the previous inequality goes to zero as $T\to \infty$. Indeed, to see this we use the power decay of $v_\rho$ as $y\to \infty$ given by Proposition~\ref{Prop:GradientEstimate}, and the bounds for $|\nabla (\psi_T) |$, $\eta_\varepsilon$ and $|\nabla (\rho  \eta_\varepsilon  )|$. Hence, letting $T\to \infty$ in the previous expression, we obtain
\begin{equation}
\label{Eq:InequalityTestFunctionsLipschitz}
(n-1) \int_0^{\infty} \int_{B_1} y^a v_\rho^2 \eta_\varepsilon^2  \d x \d y 
\leq  
\int_0^{\infty} \int_{B_1}  y^a v_\rho^2 |\nabla (\rho  \eta_\varepsilon ) |^2 \d x \d y\,,
\end{equation}
for every $\eta_\varepsilon(\rho) \in C^1(B_1)$ with compact support and such that  $|\nabla (\rho  \eta_\varepsilon  )|$ is bounded. By approximmation, $\eta_\varepsilon$ can be taken to be Lipschitz instead of $C^1$.

Now, for $\varepsilon \in (0,1/2)$ and $\alpha$ satisfying \eqref{Eq:AlphaCondition}, we define
$$
\eta_\varepsilon (\rho) =
\begin{cases}
\varepsilon^{-\alpha} & \textrm{ if } 0 \leq \rho \leq \varepsilon\,, \\
\rho^{-\alpha}\varphi(\rho)  & \textrm{ if } \varepsilon \leq \rho \,, 
\end{cases}
$$
where $\varphi \geq 0$ is a smooth cut-off function such that $\varphi(\rho) \equiv 1$ if $\rho \leq 1/2$ and $\varphi(\rho) \equiv 0$ if $\rho \geq 3/4$. Taking $\eta_\varepsilon$ in \eqref{Eq:InequalityTestFunctionsLipschitz} and using that $\varphi \geq 0$, we get
\begin{align}
\label{Eq:MissingTermCapellaEtAl}
(n-1)  \int_0^{\infty} \int_{B_{1/2} \setminus B_\varepsilon} y^a v_\rho^2 \rho^{-2\alpha}  \d x \d y 
+ 
(n-1) \varepsilon^{-2\alpha}  \int_0^{\infty} \int_{B_\varepsilon} y^a v_\rho^2 \d x \d y \hspace{-115mm} \nonumber \\
&\leq
(1-\alpha)^2  \int_0^{\infty} \int_{B_{1/2} \setminus B_\varepsilon} y^a v_\rho^2 \rho^{-2\alpha}  \d x \d y 
+ 
\varepsilon^{-2\alpha}  \int_0^{\infty} \int_{B_\varepsilon} y^a v_\rho^2 \d x \d y \nonumber \\
& \quad \quad \quad + C  \int_0^{\infty} \int_{B_{3/4} \setminus B_{1/2}} y^a v_\rho^2 \rho^{-2\alpha}  \d x \d y \,,
\end{align} 
for a constant $C$ depending only on $\alpha$ and $n$. Since $n \geq 2$ and $\alpha$ satisfies \eqref{Eq:AlphaCondition}, we obtain
\begin{equation}
\label{Eq:EstimateProposition1.3}
 \int_0^{\infty} \int_{B_{1/2} \setminus B_\varepsilon} y^a v_\rho^2 \rho^{-2\alpha}  \d x \d y  \leq C  \int_0^{\infty} \int_{B_{3/4} \setminus B_{1/2}} y^a v_\rho^2 \rho^{-2\alpha}  \d x \d y\,,
\end{equation}
for another constant depending only on $n$ and $\alpha$. Finally, we estimate the right hand side of this last inequality using the estimates developed in Section~\ref{Sec:PreliminaryResults}. To do this, we split the integral into two parts:
$$
\int_0^1    \int_{B_{3/4}\setminus B_{1/2}}  y^a  v_\rho^2 \rho^{-2\alpha} \d x \d y + \int_1^{\infty}   \int_{B_{3/4}\setminus B_{1/2}}  y^a  v_\rho^2 \rho^{-2\alpha}\d x \d y \,.
$$
We bound the first term using Proposition~\ref{Prop:HorizontalGradientEstimateRing}, obtaining:
$$
\int_0^{1}    \int_{B_{3/4}\setminus B_{1/2}} y^a  v_\rho^2  \rho^{-2\alpha} \d x \d y \leq C \bpar{\int_0^1 y^a \d y } \bpar{\int_{B_{3/4} \setminus B_{1/2}} \rho^{-2\alpha} \d x}\leq C
$$
where the last constant $C$ depends only on $n$, $s$, $\alpha$, $\norm{u}_{L^{1}(B_1)}$, $\norm{u}_{L^{\infty}(B_{7/8}\setminus B_{3/8})}$ and $\norm{f'(u)}_{L^{\infty}(B_{7/8}\setminus B_{3/8})}$.
For the second term, we use the uniform estimate $|\nabla_x v|\leq C \norm{u}_{L^{1}(B_1)}/y^{n+1}$, given by Proposition~\ref{Prop:GradientEstimate}, to get
$$
\int_1^{\infty}   \int_{B_{3/4}\setminus B_{1/2}}  y^a  v_\rho^2 \rho^{-2\alpha} \d x \d y \leq C \norm{u}_{L^{1}(B_1) }^2 \int_1^{\infty}   \int_{B_{3/4}\setminus B_{1/2}} y^{a-2n-2} \rho^{-2\alpha} \d x \d y  \leq C,
$$
for a constant $C$ depending only on $n$, $s$, $\alpha$ and $\norm{u}_{L^{1}(B_1)}$.

Finally, using these estimates in \eqref{Eq:EstimateProposition1.3} and letting $\varepsilon \to 0$, we conclude the proof.
\end{proof}

\begin{remark}
\label{Remark:MissingTermCapellaEtAl}

As mentioned in the introduction, in the proof of Proposition~5.1 of \cite{CapellaEtAl} ---which is similar to the previous one---, there is a missing term which remains to be estimated. This is the one appearing in \eqref{Eq:MissingTermCapellaEtAl}, but with a different power of $\rho$. In the case of the spectral fractional Laplacian, the estimate we need is given by Proposition~\ref{Prop:HorizontalGradientEstimateRing}, which is valid for both operators $A^s$ and $\fraclaplacian$ (see Remark~\ref{Remark:EstimateValidForBothOperators}). Therefore, the proof of Proposition 5.1 of \cite{CapellaEtAl} is now complete.

\end{remark}

With a small modification of the previous proof, we can replace the constant on the right-hand side of \eqref{Eq:UniformBoundHalfBall} by $C \norm{u}_{H^s(\R^n)}$ with $C$ depending only on $n$, $s$ and $\alpha$.

\begin{proposition}
	\label{Prop:IntegrabilityHalfBallHs}
	Under the same hypotheses of Proposition~\ref{Prop:IntegrabilityHalfBall}, we have
	\begin{equation}
	\label{Eq:UniformBoundHalfBallHs}
	\int_0^{\infty}
	\int_{B_{1/2}} y^a v_\rho^2 \rho^{-2\alpha} \d x \d y \leq C \seminorm{u}_{H^s(\R^n)}\,,
	\end{equation}
	where $C$ is a constant which depends only on $n$, $s$ and $\alpha$.
\end{proposition}

\begin{proof}
	We follow the previous proof up to \eqref{Eq:EstimateProposition1.3} and then we use that
	$$
	\seminorm{v}_{H^1(\R_+^{n+1}, y^a)} = \dfrac{1}{d_s} \seminorm{u}_{H^s(\R^n)}\,.
	$$
	This follows from the fact that $v$ solves $\div(y^a \nabla v) = 0$ in $\R_+^{n+1}$ (see  Section~\ref{Sec:ExtensionProblem}).	
\end{proof}

\begin{remark}
\label{Remark:SharpfRegularity}
The hypotheses for $f$ in Proposition~\ref{Prop:IntegrabilityHalfBall} ---and also in Theorem~\ref{Th:RadialMainTheorem}--- can be slightly weakened. Indeed, the statements remain true if, instead of $f$ being $C^{2}$ we assume that $f \in C^{2 - 2s + \varepsilon}([0, +\infty))$ for $\varepsilon >0$. In particular, for $s > 1/2$, it is enough to assume $f\in C^1$. This regularity is needed in order to have $u\in C^2_{\textrm{loc}}(B_1\setminus \{0\})$, a fact that is used in the previous proofs.
\end{remark}

\section{Proof of the main theorem}
\label{Sec:ProofMainTheorem}

In this section we prove Theorem~\ref{Th:RadialMainTheorem}. As explained before the statement of Proposition~\ref{Prop:IntegrabilityHalfBall}, to get an $L^\infty$ bound for $u$ we still need a crucial identity and a precise bound on a universal constant. This is the content of Lemma~6.2 in \cite{CapellaEtAl}. We include it here with a slightly different statement and proof that probably make the result and proof more transparent.

\begin{lemma}
\label{Lemma:Identityv_yMagicConstant}
Let $w:\R^n \rightarrow \R$ be a bounded function with compact support and such that $\fraclaplacian w \in L^\infty_{\mathrm{loc}}(\R^n)$. Let $W$ be its $s$-harmonic extension and let $\beta$ be a real number such that $0<\beta < n + 2 - 2s$.
Then,
$$
-d_s \beta \int_{\R_+^{n+1}} y^a r^{-\beta-2} y W_y \d x \d y = A_{n,s,\beta} \int_{\R^n} \rho^{-\beta}  \fraclaplacian w \d x\,,
$$
for a constant $A_{n,s,\beta}$ depending only on $n$, $s$, and $\beta$ and satisfying 
$$
0 < A_{n,s,\beta} < 1\,.
$$
\end{lemma}

\begin{proof}
Consider the following two operators:
$$
\mathcal{F}^\varepsilon_\beta (w) := - d_s \beta \int_{\R_+^{n+1}} y^a (|x|^2 + y^2 + \varepsilon)^{-(\beta + 2)/2} y W_y \d x \d y \,,
$$
$$
\mathcal{G}_\beta (w) = \int_{\R^n} \rho^{-\beta}  \fraclaplacian w \d x\,.
$$
First, we will show that $ \lim_{\varepsilon \to 0} \mathcal{F}^\varepsilon_\beta (w) = A_{n,s,\beta} \mathcal{G}_\beta (w)$ and later we will see that we have $0 < A_{n,s,\beta} < 1$.

Using the Poisson formula for $W_y$ (Lemma~\ref{Lemma:PoissonFormulaConjugate}), we find that
$$
-d_s W_y(x,y) = \Gamma_{n,s} y \int_{\R^n}  \dfrac{\fraclaplacian w(z)}{(|x-z|^2 + y^2)^\frac{n+2-2s}{2}}    \d z \,.
$$
Now, multiply the previous equation by $\beta y^{a+1} (|x|^2 + y^2 + \varepsilon)^{-(\beta+2)/2}$ and integrate in the whole $\R_+^{n+1}$ to obtain
\begin{align*}
\mathcal{F}^\varepsilon_\beta (w) = \quad \hspace{380pt} &\\ 
= \int_{\R^n}  \fraclaplacian w(z)\bpar{ \beta \Gamma_{n,s}  \int_{\R_+^{n+1}} \dfrac{y^{a+2}}{(|x-z|^2 + y^2)^{\frac{n+2-2s}{2}} (|x|^2 + y^2 + \varepsilon)^{\beta/2 + 1} } \d x \d y} \d z \,.
\end{align*}

After the change of variables  $x=|z|x'$, $y=|z|y'$, we get
$$
\mathcal{F}^\varepsilon_\beta (w) = \int_{\R^n} \fraclaplacian w(z) |z|^{-\beta} A_{n,s,\beta} \bpar{\dfrac{\varepsilon}{ |z|^2}} \d z\,,
$$
where
$$
A_{n,s,\beta} (t) = \beta \Gamma_{n,s}\int_{\R_+^{n+1}} 
\dfrac{y^{a+2}}{(\rho^2 + y^2 + t)^{\frac{\beta +2}{2}}(|x-\frac{z}{|z|}|^2 + y^2)^{\frac{n+2-2s}{2}}}  \d x \d y \,.
$$
Notice that $A_{n,s,\beta} (t)$ does not depend on $z$ and that
$$
 A_{n,s,\beta} \bpar{\dfrac{\varepsilon}{ |z|^2}} \rightarrow A_{n,s,\beta} :=  A_{n,s,\beta} (0) \quad \text{ for all } z \in \R^n
$$
as $\varepsilon \to 0$. Moreover, this limit is finite for $0<\beta < n + 2 - 2s$.
Hence, we have proved that 
$$
\mathcal{F}_\beta (w) := \lim_{\varepsilon \to 0} \mathcal{F}^\varepsilon_\beta (w) = A_{n,s,\beta} \mathcal{G}_\beta (w)\,,
$$
with a nonnegative constant $A_{n,s,\beta}$ given by
$$
A_{n,s,\beta} = \beta \Gamma_{n,s}\int_{\R_+^{n+1}} 
\dfrac{y^{a+2}}{(\rho^2 + y^2)^{\frac{\beta +2}{2}}(|x-e|^2 + y^2)^{\frac{n+2-2s}{2}}}  \d x \d y
$$
for an arbitrary unitary vector $e$.

Now, let us prove that the constant $A_{n,s,\beta}$ is smaller than one. Take $h\in C^\infty (\R^n) $, $h \not \equiv 0$, a smooth nonnegative radially decreasing function with compact support. Let $w \geq 0 $ be the solution of $\fraclaplacian w = h$ in $\R^n$ and let $W$ be its $s$-harmonic extension. Note that, by the moving planes argument, $w$ is radially decreasing and so it is $W$ in the horizontal direction by the Poisson formula.

Take the equation that $W$ satisfies, that is, $\div(y^a \nabla W) = 0$ and multiply it by $d_s r^{-\beta} = d_s (|x|^2 + y^2)^{-\beta/2}$. After integration by parts we find that
\begin{align*}
0 &= d_s \int_{\R_+^{n+1}} \div(y^a \nabla W)r^{-\beta} \d x \d y = \\
& = \beta d_s \int_{\R_+^{n+1}} y^a (\rho W_\rho + y W_y) r^{-\beta -2} \d x \d y + \int_{\R^n} \rho^{-\beta} \fraclaplacian w \d x\,.
\end{align*}
Therefore, we have
$$
\dfrac{1}{A_{n,s,\beta}}  \mathcal{F}_\beta (w) = \mathcal{G}_\beta (w) = -\beta d_s \int_{\R_+^{n+1}} y^a \rho W_\rho r^{-\beta -2} \d x \d y + \mathcal{F}_\beta (w)\,,
$$
which leads to
$$
\bpar{\dfrac{1}{A_{n,s,\beta}} - 1} \mathcal{F}_\beta (w) = - \beta d_s \int_{\R_+^{n+1}} y^a \rho W_\rho r^{-\beta -2} \d x \d y > 0\,,
$$
since $W$ is radially decreasing, i.e., $W_\rho < 0$.
This leads to  $0<A_{n,s,\beta} < 1$.
\end{proof}

Once this lemma is established, we have all the ingredients to present the proof of our main result:

\begin{proof}[Proof of Theorem~\ref{Th:RadialMainTheorem}]
We divide our proof into two steps.

\textbf{Step 1.}
We claim that, for $\alpha$ satisfying \eqref{Eq:AlphaCondition} and $\beta > 0$ a real number such that $2(\beta +s-\alpha) < n$, 
\begin{equation}
    \label{Eq:IntegralRadialBeta}
    \int_{B_1} f(u) \rho^{-\beta} \d x \leq C
\end{equation}
with a constant $C$ which depends only on $n$, $s$, $\alpha$, $\beta$, $\norm{u}_{L^1(B_1)}$,  $\norm{u}_{L^{\infty}(B_{7/8}\setminus B_{3/8})}$, $\norm{f(u)}_{L^\infty(B_1\setminus B_{1/2})}$ and $\norm{f'(u)}_{L^{\infty}(B_{7/8}\setminus B_{3/8})}$.

To prove the claim, we first multiply $\div (y^a \nabla v) =  0$ by $d_s(\rho^2 + y^2 + \varepsilon)^{-\beta/2}$ and integrate it in the cylinder $B_{1/2}\times (0,1)$. We get
\begin{align*}
0 & = d_s \int_{B_{1/2}\times (0,1)}  \div (y^a \nabla v )(\rho^2 + y^2 + \varepsilon)^{-\beta/2} \dx \d y\\
&= \int_{B_{1/2}} f(u) (\rho^2 +\varepsilon)^{-\beta/2} \dx
+ d_s \int_{B_{1/2}} v_y(\rho, 1) (\rho^2 +1 + \varepsilon)^{-\beta/2} \dx \\
&\quad \quad + d_s \int_0^1 y^a v_\rho(1/2,y)(1/4 + y^2 + \varepsilon)^{-\beta/2} \d y \\ 
&\quad \quad + d_s \beta  \int_{B_{1/2}\times (0,1)}  y^a (\rho^2 + y^2 + \varepsilon)^{-\beta/2-1} (\rho v_\rho + y v_y)\dx \d y\,.
\end{align*}
We rewrite this as
\begin{equation}
\label{Eq:IntegralDecomposition}
	\int_{B_{1/2}} f(u)(\rho^2 +\varepsilon)^{-\beta/2} \dx = -I_1 - I_2 + I_3\,,
\end{equation}
where
$$
I_1= d_s \int_{B_{1/2}} v_y(\rho,1) (\rho^2 + 1 + \varepsilon)^{-\beta/2} \dx  \,, 
$$
$$
I_2 = d_s \int_0^1 y^a v_\rho(1/2,y)(1/4 + y^2 + \varepsilon)^{-\beta/2} \d y \,,
$$
and
$$
I_3 = -d_s \beta  \int_0^1\int_{B_{1/2}}   y^a (\rho^2 + y^2 + \varepsilon)^{-\beta/2-1} (\rho v_\rho + y v_y)\dx \d y  \,.
$$

We decompose $I_3 = I_\rho + I_y$, where
$$
I_\rho = -d_s \beta  \int_0^1\int_{B_{1/2}}  y^a (\rho^2 + y^2 + \varepsilon)^{-\beta/2-1} \rho v_\rho\dx \d y 
$$
and
$$
I_y = -d_s \beta  \int_0^1\int_{B_{1/2}}  y^a (\rho^2 + y^2 + \varepsilon)^{-\beta/2-1}  y v_y\dx \d y \,.
$$
We can estimate $ \lim_{\varepsilon \to 0} I_y$ following the arguments of Lemma~\ref{Lemma:Identityv_yMagicConstant} to obtain
\begin{equation}
	\label{Eq:BoundI_y}
	\lim_{\varepsilon \to 0} I_y \leq A_{n,s,\beta} \int_{B_1} f(u) |x|^{-\beta} \d x\,,
\end{equation}
where $A_{n,s,\beta}$ is the constant appearing in Lemma~\ref{Lemma:Identityv_yMagicConstant}. Recall that by this lemma,  $0<A_{n,s,\beta}<1$. Indeed, we have that 
\begin{align*}
\lim_{\varepsilon \to 0} I_y &= \int_{\R^n}  \fraclaplacian u(z)\bpar{ \beta \Gamma_{n,s} \int_0^1\int_{B_{1/2}} \dfrac{y^{a+2}}{(|x-z|^2 + y^2)^{\frac{n+2-2s}{2}} (|x|^2 + y^2)^{\beta/2 + 1} } \d x \d y} \!\! \d z \\
&\leq 
\int_{B_1}  f(u) \bpar{ \beta \Gamma_{n,s}  \int_0^1\int_{B_{1/2}} \dfrac{y^{a+2}}{(|x-z|^2 + y^2)^{\frac{n+2-2s}{2}} (|x|^2 + y^2 )^{\beta/2 + 1} } \d x \d y} \d z \\
&\leq 
A_{n,s,\beta} \int_{B_1} f(u) |x|^{-\beta} \d x\,.
\end{align*}
Here we have used the Poisson formula for $v_y$ 
in the first equality. Then, we have used that $\fraclaplacian u < 0$ in $\R^n \setminus B_1$ and also the equation $\fraclaplacian u = f(u)$ in $B_1$. The last inequality is easily deduced using exactly the same arguments that are described in the proof of  Lemma~\ref{Lemma:Identityv_yMagicConstant}.

From \eqref{Eq:IntegralDecomposition} and \eqref{Eq:BoundI_y}, we deduce that
\begin{align*}
 \int_{B_1} f(u) \rho^{-\beta} \d x &= \int_{B_1\setminus B_{1/2}} f(u) \rho^{-\beta} \d x
 + \int_{B_{1/2}} f(u) \rho^{-\beta} \d x\\
 & \leq C_{n, \beta} \norm{f(u)}_{L^\infty(B_1\setminus B_{1/2})} + \limsup_{\varepsilon\rightarrow 0} \bpar{|I_1| + |I_2| + |I_\rho|} \\
 & \quad \quad + 
 A_{n,s,\beta} \int_{B_1} f(u) \rho^{-\beta} \d x \,.
\end{align*}
Since $u$ is radially decreasing, $f(u) \rho^{-\beta}$ is bounded in $B_1\setminus B_{1/2}$. Thus, we obtain
$$
\bpar{1- A_{n,s,\beta}} \int_{B_1} f(u) \rho^{-\beta} \d x \leq C + \limsup_{\varepsilon\rightarrow 0} \bpar{|I_1| + |I_2| + |I_\rho|} \,,
$$
for a constant $C$ depending only on $n$, $\beta$ and $\norm{f(u)}_{L^\infty(B_1\setminus B_{1/2})} $.
Moreover, thanks to Lemma~\ref{Lemma:Identityv_yMagicConstant}, $1 - A_{n,s,\beta} > 0$ and therefore
$$
\int_{B_1} f(u) \rho^{-\beta} \d x \leq C (1 +\limsup_{\varepsilon\rightarrow 0} \bpar{|I_1| + |I_2| + |I_\rho|}) \,,
$$
with a constant $C$ depending only on $n$, $s$, $\beta$ and $\norm{f(u)}_{L^\infty(B_1\setminus B_{1/2})} $.

Hence, in order to prove our claim, we only need to bound $\limsup_{\varepsilon\rightarrow 0} \bpar{|I_1| + |I_2| + |I_\rho|}$. This is done using some previous results, as follows.

We first bound $|I_1|$. Since this integral is computed over $B_{1/2} \times \{1\}$, we can use the gradient estimate $|\nabla v| \leq C$ (see Proposition~\ref{Prop:GradientEstimate}) with a constant $C$ depending only on $n$, $s$ and $\norm{u}_{L^1(\R^n)}$. 

For $|I_2|$, we just use Proposition~\ref{Prop:HorizontalGradientEstimateRing} to bound $|v_\rho|$ in $\{\rho=1/2\} \times (0,1)$ by a constant depending only on $n$, $s$, $\norm{u}_{L^1(B_1)}$,  $\norm{u}_{L^{\infty}(B_{7/8}\setminus B_{3/8})}$ and $\norm{f'(u)}_{L^{\infty}(B_{7/8}\setminus B_{3/8})}$.

Finally, for $|I_\rho|$, using the Cauchy-Schwarz inequality we get
$$
|I_\rho| \leq d_s \beta \bpar{\int_{B_{1/2}\times (0,1)}  y^a \rho^{-2\alpha} v_\rho^2\dx \d y }^{1/2}\bpar{\int_{B_{1/2}\times (0,1)} \dfrac{y^a \rho^{2+2\alpha}}{ (\rho^2 + y^2 + \varepsilon)^{\beta+2}}  \dx \d y }^{1/2} .
$$
The first of these integrals is bounded by a constant which depends only on $\alpha$ and on the same quantities as the previous one, thanks to Proposition~\ref{Prop:IntegrabilityHalfBall}. To see that the second integral is finite, we notice that
\begin{align*}
    \int_{B_{1/2}\times (0,1)} \dfrac{y^a \rho^{2+2\alpha}}{ (\rho^2 + y^2 + \varepsilon)^{\beta+2}}  \dx \d y 
    & \leq \int_0^\infty \int_{B_{1/2}} \dfrac{y^a \rho^{2\alpha}}{(\rho^2 + y^2)^{\beta+1}} \d x \d y \\
    &= \bigg ( \int_{B_{1/2}} \rho^{a+2\alpha-2\beta - 1} \d x \bigg ) \bigg ( \int_0^\infty \dfrac{t^a}{(1+t)^{\beta + 1}} \d t \bigg ) \,,
\end{align*}
where we have made the change $y=\rho t$. These integrals are finite if $\beta > 0$ and $n>2(\beta + s - \alpha)$ ---recall that $a = 1 - 2s$. Therefore, the claim \eqref{Eq:IntegralRadialBeta} is proved.

\medskip

\textbf{Step 2.} 
We prove point (i) of the statement of the theorem. Thanks to the representation formula for the fractional Laplacian and the fact that $u$ is radially decreasing, it is easy to see that
\begin{equation}
\label{Eq:LinftyNormEstimateProofMainTheorem}
	\norm{u}_{L^{\infty}(B_1)} = u(0) \leq C \int_{B_1} \dfrac{f(u(x))}{|x|^{n-2s}} \d x\,,
\end{equation}
where $C$ is a constant depending only on $n$ and $s$. Indeed, we just use Lemma~\ref{Lemma:Boundu(x)Dirichlet} with a truncation of $f(u)$ (recall that in such lemma $h$ is assumed to be bounded) and then use monotone convergence to deduce \eqref{Eq:LinftyNormEstimateProofMainTheorem}. 
In order to use the claim of Step~1, we take $\beta = n - 2s$ and we must choose $\alpha$ satisfying $2(\beta + s -\alpha)<n$ and $1 \leq \alpha < 1 + \sqrt{n-1}$. Therefore, we require that $n/2 - s < \alpha$ (thus $1 \leq \alpha$ provided that $n\geq 2$) and $\alpha < 1 + \sqrt{n-1}$. Hence, such $\alpha$ exists if and only if $n/2 - s <1 + \sqrt{n-1}$, which is equivalent to
\begin{equation}
    \label{Eq:ConditionOnNRadial}
     2\bpar{s+2-\sqrt{2(s+1)}} < n < 2\bpar{s+2+\sqrt{2(s+1)}}.
\end{equation}
Notice that the lower bound for $n$ is automatically satisfied for $n\geq 2$ and $s\in(0,1)$. Thus, if $n$ satisfies \eqref{Eq:ConditionOnNRadial}, we can take $\alpha$ such that \eqref{Eq:IntegralRadialBeta} holds for $\beta = n - 2s$. Therefore, by \eqref{Eq:LinftyNormEstimateProofMainTheorem} and Step~1, we obtain
$$
\norm{u}_{L^{\infty}(B_1)}\leq C
$$
with a constant $C$ depending only on $n$, $s$, $\norm{u}_{L^1(B_1)}$,  $\norm{u}_{L^{\infty}(B_{7/8}\setminus B_{3/8})}$, $\norm{f(u)}_{L^{\infty}(B_1\setminus B_{1/2})}$ and $\norm{f'(u)}_{L^{\infty}(B_{7/8}\setminus B_{3/8})}$.
Next, we replace this $C$ by another constant depending only on $n$, $s$, $f$ and $\norm{u}_{L^1(B_1)}$. To do this, we control the $L^\infty$ norm of $u$ in sets away from the origin by the $L^1$ norm of $u$. Indeed, since $u$ is radially decreasing, we have that
$$
\norm{u}_{L^\infty (B_1\setminus \overline{B_R})} \leq \dfrac{C}{R^n} \norm{u}_{L^1(B_1)}   \quad \textmd{for every } R \in (0,1) \,.
$$

\medskip

Finally, we prove (ii). Assume that $\alpha$ and $\beta$ satisfy the  hypotheses of Step~1. Then, using that $f$ is nondecreasing, that $u$ is radially decreasing, and \eqref{Eq:IntegralRadialBeta}, we have
$$
c_n \rho^{n-\beta}f(u (\rho)) = f(u (\rho)) \int_{B_{2\rho}\setminus B_\rho} \!\! |x|^{-\beta} \d x \leq \! \int_{B_1} f(u) |x|^{-\beta} \d x \leq C \ \  \text{ for } \rho \leq 1/2 .
$$
Therefore,
\begin{equation}
\label{Eq:Boundf(u)beta}
	f(u (\rho))\leq  C\rho^{\beta-n} \ \ \text{ for } 0< \rho \leq 1\,,
\end{equation}
with a constant $C$ depending only on $n$, $s$, $\norm{u}_{L^1(B_1)}$,  $\norm{u}_{L^{\infty}(B_{7/8}\setminus B_{3/8})}$, $\norm{f(u)}_{L^{\infty}(B_1\setminus B_{1/2})}$ and $\norm{f'(u)}_{L^{\infty}(B_{7/8}\setminus B_{3/8})}$.

Assume additionally that $\beta < n - 2s$. Using Lemma~\ref{Lemma:Boundu(x)Dirichlet} and \eqref{Eq:Boundf(u)beta}, we obtain
$$
u (x) \leq \dfrac{C}{|x|^{n-\beta - 2s}}\ \ \text{ for all } x \in B_1\,.
$$
From the restrictions on $\alpha$ and $\beta$ that we assumed, we conclude that for every $\mu$ with $\mu > n/2 - s - 1 - \sqrt{n-1}$, we have
$$
u (x) \leq \dfrac{C}{|x|^{\mu}}\ \ \text{ for all } x \in B_1\,,
$$
for a constant $C$ depending only on $n$, $s$, $\mu$, $\norm{u}_{L^1(B_1)}$,  $\norm{u}_{L^{\infty}(B_{7/8}\setminus B_{3/8})}$, $\norm{f(u)}_{L^{\infty}(B_1\setminus B_{1/2})}$ and $\norm{f'(u)}_{L^{\infty}(B_{7/8}\setminus B_{3/8})}$.
As before, using that $u$ is radially decreasing we can deduce the same estimate but with a constant $C$ depending only on  $n$, $s$, $\mu$, $f$ and $\norm{u}_{L^1(B_1)}$. 
\end{proof}

\begin{remark}
\label{Remark:MissingEstimateCapellaEtAl}
In \cite{CapellaEtAl} there is a mistake in the proof of their analogous theorem (Theorem~1.6 there). The authors state that the integral $I_2$ can be controlled using an estimate that only holds for $y$ away from $\{y = 0\}$. Since $I_2$ is an integral up to $\{y = 0\}$, a bound for $I_2$ requires an additional argument. As we show in our proof, the proper way to bound it is by using Proposition~\ref{Prop:HorizontalGradientEstimateRing}, which is valid also for the spectral fractional Laplacian (see Remark~\ref{Remark:EstimateValidForBothOperators}).
\end{remark}

We conclude by applying the previous result to show the boundedness of the extremal solution $u^*$.

\begin{proof}[Proof of Theorem~\ref{Th:RadialMainTheoremExtremal}]
	First, note that the estimate given in point (i) of  Theorem~\ref{Th:RadialMainTheorem} is valid for the classical stable solutions $u_\lambda$ for $\lambda < \lambda^*$. This is because, obviously, $u_\lambda \in H^s(\R^n)$ and, since $u_\lambda$ are bounded and positive, they are also radially decreasing (see Proposition~\ref{Prop:RadialSymmetryAndMonotonicity}). Therefore, by Theorem~\ref{Th:RadialMainTheorem}, we have
	$$
	\norm{u_\lambda}_{L^\infty (B_1)} \leq C
	$$
	for some constant $C$ depending only on $n$, $s$, $f$ and $\norm{u_\lambda}_{L^1 (\R^n)}$. Note that all these quantities are uniform in $\lambda < \lambda^*$ (see Remark~\ref{Remark:UniformEstimatesInLambda}). Hence, by letting $\lambda \to \lambda^*$ we conclude
	$$
	\norm{u^*}_{L^\infty (B_1)} \leq C
	$$
	for some constant $C$ depending only on $n$, $s$, $f$ and $\norm{u^*}_{L^1 (\R^n)}$.
	
	The way to deduce point (ii) from  Theorem~\ref{Th:RadialMainTheorem} is completely analogous.
\end{proof}

\appendix

\section{An alternative proof of the result of Ros-Oton and Serra for the exponential nonlinearity}
\label{Sec:GelfandProblem}

In this appendix, we present an alternative proof of the following result of X.~Ros-Oton and J.~Serra. In contrast with theirs, our proof uses the extension problem.

\begin{proposition}[Proposition 3.1 in \cite{RosOtonSerra-Extremal}]
\label{Prop:GelfandProblem10s}
Let $\Omega$ be a smooth and bounded domain in $\R^n$, and let $u^*$ be the extremal solution of \eqref{Eq:ProblemLambda}. Assume that $f(u) = e^u$ and $n < 10 s$. Then, $u^*$ is bounded.
\end{proposition}

The procedure used to prove the boundedness of the extremal solution is, as usual, to deduce an $L^{\infty}$ estimate for $u_\lambda$ uniform in $\lambda < \lambda^*$. Then, the result follows from monotone convergence. To prove the uniform bound for $u_\lambda$, we proceed as in the classical proof of Crandall-Rabinowitz \cite{CrandallRabinowitz}: we take $\xi = e^{\alpha u_\lambda} - 1$ in the stability condition to obtain a uniform $L^{p}$ bound for $e^{u_\lambda}$ for certain values of $p$. This, combined with the following result, will lead to the desired $L^{\infty}$ estimate.

\begin{lemma}[\cite{RosOtonSerra-Extremal}]
\label{Lemma:LpEmbeddings}

Let $\Omega \subset \R^n$ be  a  bounded $C^{1,1}$ domain,
$s\in (0,1)$, $n > 2s$, $h\in C(\overline{\Omega})$, and $u$
be the solution of 
$$
\beqc{\PDEsystem}
    \fraclaplacian u &= & h & \textrm{in } \Omega\,,\\
    u &= &0 &  \textrm{in } \R^n\setminus \Omega\,.
\eeqc
$$
Let $\frac{n}{2s} < p < \infty$. Then, there exists a constant $C$, depending only on $n$, $s$, $p$ and $|\Omega|$, such that
$$
\norm{u}_{C^{\beta}(\R^n)} \leq C \norm{h}_{L^{p}(\Omega)}\,, \quad \ \textrm{ where } \ \beta =  \min \left \{ s, 2s - \dfrac{n}{p} \right \} \,.
$$
\end{lemma}

With this bound in hand, we can proceed with the alternative proof of the result on the boundedness of $u^*$ in the case $f(u)=e^u$.

\begin{proof}[Proof of Proposition~\ref{Prop:GelfandProblem10s}]
Let $\alpha$ be a positive real number that will be chosen later. Let $u_\lambda$ be the minimal stable solution of \eqref{Eq:ProblemLambda} for $\lambda < \lambda^*$. Take $\xi = e^{\alpha u_\lambda} - 1$, which is $0$ in $\R^n\setminus \Omega$, in the stability condition \eqref{Eq:StabilityConditionExtension} to obtain
$$
\lambda \int_{\Omega} e^{u_\lambda} (e^{\alpha u_\lambda} - 1)^2 \d x \leq d_s \int_{\R_+^{n+1}} y^a \alpha^2 e^{2\alpha v_\lambda} |\nabla v_\lambda |^2 \d x \d y \,,
$$
where $v_\lambda$ denotes the $s$-harmonic extension of $u_\lambda$. Note that we have taken $e^{\alpha v_\lambda} - 1$ as the extension of $\xi$ in $\R^{n+1}_+$. Then, integrating by parts we compute
\begin{align*}
d_s \int_{\R_+^{n+1}} y^a \alpha^2 e^{2\alpha v_\lambda} |\nabla v_\lambda |^2 \d x \d y &= d_s \dfrac{\alpha}{2}  \int_{\R_+^{n+1}} y^a \nabla v_\lambda \cdot \nabla (e^{2\alpha v_\lambda} - 1) \d x \d y \\ &=\dfrac{\alpha}{2} \int_{\Omega} \lambda e^{u_\lambda} (e^{2\alpha u_\lambda} - 1) \d x
\end{align*}
(recall that $\div(y^a \nabla v_\lambda ) = 0$ in $\R_+^{n+1}$) and hence
$$
\int_{\Omega} e^{u_\lambda} (e^{2\alpha u_\lambda} - 2 e^{\alpha u_\lambda} + 1) \d x \leq \dfrac{\alpha}{2} \int_{\Omega} e^{u_\lambda} (e^{2\alpha u_\lambda} - 1) \d x\,.
$$
This leads to
\begin{equation}
\label{Eq:ExponentialInequalitiesAlpha}
\bpar{1 - \dfrac{\alpha}{2}}
\int_{\Omega} e^{(2\alpha + 1)u_\lambda} \d x
- 2
\int_{\Omega} e^{(\alpha + 1)u_\lambda} \d x
+ \bpar{1 + \dfrac{\alpha}{2}}
\int_{\Omega} e^{u_\lambda} \d x \leq 0\,.
\end{equation}

Now, using Hölder inequality we have
$$
\int_{\Omega} e^{(\alpha + 1)u_\lambda} \d x \leq C \bpar{\int_{\Omega} e^{(2\alpha + 1)u_\lambda} \d x}^{\frac{\alpha + 1}{2 \alpha + 1}}
$$
for a constant $C$ depending only on $\alpha$ and $|\Omega|$. Therefore, from \eqref{Eq:ExponentialInequalitiesAlpha} we see that for each $\alpha < 2$, we have
\begin{equation}
\label{Eq:ExponentialUniformLpBound}
\norm{e^{u_\lambda}}_{L^{2\alpha + 1}(\Omega)} \leq C
\end{equation}
for a constant $C$ depending only on $\alpha$ and $|\Omega|$.

Finally, if $n < 10s$,
we can choose $\alpha  < 2$ such that  $\frac{n}{2s} < 2 \alpha + 1  < 5$. Then, taking $p = 2\alpha+ 1$ in Lemma~\ref{Lemma:LpEmbeddings} and using \eqref{Eq:ExponentialUniformLpBound} we obtain
$$
\norm{u_\lambda}_{L^{\infty}(\Omega)} \leq C \lambda \norm{e^{u_\lambda}}_{L^{2\alpha + 1}(\Omega)} \leq C\,,
$$
for a constant $C$ depending only on $n$, $s$ and $\Omega$ (and hence independent of $\lambda$). By monotone convergence, letting $\lambda \to \lambda^*$ we conclude that $u^*$ is bounded.
\end{proof}

\begin{remark}
In the previous proof, we have taken $\xi = e^{\alpha v_\lambda} - 1$ in the stability condition, where $v_\lambda$ is the $s$-harmonic extension of $u_\lambda$. Nevertheless, the inequality obtained with this choice of the extension is not sharp, since $e^{\alpha v_\lambda} - 1$ is not the $s$-harmonic extension of $\xi = e^{\alpha u_\lambda} - 1$. This choice simplifies a lot the computations but makes us wonder if there could be a smarter choice of the extension leading to a better result.
\end{remark}

\section*{Acknowledgements}

The author thanks Xavier Cabré for his guidance and useful discussions on the topic of this paper.

\bibliographystyle{amsplain}
\bibliography{biblio}

\end{document}